\documentclass[11pt]{amsart}

\usepackage{amsthm,amssymb,url,hyperref,fullpage}
\usepackage[usenames]{color}

\usepackage{tikz}
\usetikzlibrary{matrix,arrows}

\begin{document}

\numberwithin{equation}{section}

\def\comment#1#2{\textcolor{blue}{(#1: #2)}}
\def\red#1{\textcolor{red}{#1}}

\theoremstyle{plain}
\newtheorem{theorem}{Theorem}[section]
\newtheorem{conjecture}[theorem]{Conjecture}
\newtheorem{corollary}[theorem]{Corollary}
\newtheorem{definition}[theorem]{Definition}
\newtheorem{lemma}[theorem]{Lemma}
\newtheorem{proposition}[theorem]{Proposition}
\newtheorem{problem}[theorem]{Problem}

\theoremstyle{definition}
\newtheorem{example}[theorem]{Example}
\newtheorem{remark}[theorem]{Remark}
\newtheorem{algorithm}[theorem]{Algorithm}
\newtheorem{construction}[theorem]{Construction}

\def\ld{\backslash}
\def\ldd{\backslash^\cdot}
\def\rdd{/^\cdot}
\def\im#1{\mathrm{Im}#1}
\def\ker#1{\mathrm{Ker}#1}
\def\aut#1{\mathrm{Aut}#1}
\def\aff#1{\mathrm{Aff}#1}
\def\lmlt#1{\mathrm{LMlt}#1}
\def\rmlt#1{\mathrm{RMlt}#1}
\def\mlt#1{\mathrm{Mlt}#1}
\def\inn#1{\mathrm{Inn}#1}
\def\dis#1{\mathrm{Dis}#1}
\def\Z{\mathbb Z}
\def\F{\mathbb F}
\def\Q{\mathcal Q}
\newcommand\Soc{\mathrm{Soc}}
\newcommand\PSL{\mathrm{PSL}}
\newcommand\GL{\mathrm{GL}}
\newcommand\SL{\mathrm{SL}}

\title{A guide to self-distributive quasigroups, or latin quandles}

\author{David Stanovsk\'y}

\address{Department of Algebra, Faculty of Mathematics and Physics, Charles University, Prague, Czech Republic}

\address{Department of Information Systems and Mathematical Modeling, International IT University, Almaty, Kazakhstan}

\email{stanovsk@karlin.mff.cuni.cz}

\thanks{The author was partially supported by the GA\v CR grant 13-01832S.}


\date{\today}

\begin{abstract}
We present an overview of the theory of self-distributive quasigroups, both in the two-sided and one-sided cases, and relate the older results to the modern theory of quandles, to which self-distributive quasigroups are a special case. Most attention is paid to the representation results (loop isotopy, linear representation, homogeneous representation), as the main tool to investigate self-distributive quasigroups.
\end{abstract}

\maketitle

\section{Introduction}\label{sec:intro}

\subsection{The origins of self-distributivity}

\emph{Self-distributivity} is such a natural concept: given a binary operation $*$ on a set $A$, fix one parameter, say the left one, and consider the mappings $L_a(x)=a*x$, called \emph{left translations}. If all such mappings are endomorphisms of the algebraic structure $(A,*)$, the operation is called \emph{left self-distributive} (the prefix self- is usually omitted). Equationally, the property says \[a*(x*y)=(a*x)*(a*y)\] for every $a,x,y\in A$, and we see that $*$ distributes over itself. 

Self-distributivity was pinpointed already in the late 19th century works of logicians Peirce and Schr\"oder \cite{Pei,Sch}, and ever since, it keeps appearing in a natural way throughout mathematics, perhaps most notably in low dimensional topology (knot and braid invariants) \cite{Car,Deh,Nel}, in the theory of symmetric spaces \cite{Loos} and in set theory (Laver's groupoids of elementary embeddings) \cite{Deh}. Recently, Moskovich expressed an interesting statement on his blog \cite{Mos} that while associativity caters to the classical world of space and time, distributivity is, perhaps, the setting for the emerging world of information.

\medskip
\emph{Latin squares} are one of the classical topics in combinatorics. Algebraically, a latin square is represented by a binary operation, and such algebraic structures are called \emph{quasigroups}. Formally, a binary algebraic structure $(A,*)$ is called a \emph{quasigroup}, if the equations $a*x=b$ and $y*a=b$ have unique solutions $x,y$, for every $a,b\in A$. 

It is no surprise that one of the very first algebraic works fully devoted to non-associative algebraic strucures was Burstin and Mayer's 1929 paper \emph{Distributive Gruppen von endlicher Ordnung} \cite{BM} about quasigroups that are both left and right distributive. Another earliest treatise on non-associative algebraic structures was \cite{Sus1} by Sushkevich who observed that the proof of Lagrange's theorem (the one in elementary group theory) does not use associativity in full strength and discussed weaker conditions, some related to self-distributivity, that make the proof work. These pioneering works were quickly followed by others, with various motivations. For example, Frink \cite{Fri} argued that the abstract properties of the mean value are precisely those of medial idempotent quasigroups, and self-distributivity pops up again. 

The foundations of the general theory of quasigroups were laid in the 1950s and carved in stone in Bruck's book \emph{A survey of binary systems} \cite{Bru} (despite the general title, the book leans strongly towards a particular class of \emph{Moufang loops}). Ever since, self-distributive quasigroups and their generalizations played a prominent role in the theory of quasigroups, both in the Western and the Soviet schools \cite{Bel,Gal-survey,Pfl-book}. More in the Soviet one, where the dominant driving force was Belousov's program to investigate loop isotopes of various types of  quasigroups (see the list of problems at the end of the book \cite{Bel}). We refer to \cite{Pfl-history} for a more detailed historical account.

%Selfdistributivity was not unnoticed, although did not play a central role in their investigations - I think Bruck's idea was to relax associavity in various ways, rather than replace it by something completely different. 
%Nevertheless, some did notice the importance, for instance Sherman Stein (Davis): his interesting, nearly forgotten result states that, in modern terminology, latin quandles of order 4k+2 do not exist(his paper On the foundations of quasigroups, Trans. Amer. Math. Soc. 85 (1957), 228-256, contains interesting motivation and historical remarks).
%A sucessful school of quasigroups was built in Soviet Union around V. D. Belousov in 1960's. Selfdistributivity was one of the central topics of their studies. quasigroups (latin quandles) played one of the central roles in their studies (the best sources are Belousov's 1967 book and Galkin's survey paper in J. Soviet Math. 49 (1990), no. 3, 941-967.).
%I have a feeling these people were quite a few steps ahead of the algebraic topology pack (perhaps not anymore).

\medskip
\emph{Reflection} in euclidean geometry (and elsewhere) is another example of a self-distributive operation: for two points $a,b$, consider $a*b$ to be the reflection of $b$ over $a$. The equation $a*x=b$ always has a unique solution, namely, $x=a*b$, but in many cases, reflections do not yield a quasigroup operation (e.g. on a sphere). These observations, and the resulting abstraction of the notion of a reflection, can be attributed to Takasaki and his remote 1942 work \cite{Tak}, but the real advances have been made by Loos and others two decades later \cite{Loos}. The resulting notions of \emph{kei} (Takasaki), \emph{symmetric spaces} (Loos), or \emph{involutory quandles} in the modern terminology, are axiomatized by three simple algebraic properties: left distributivity, \emph{idempotence} ($a*a=a$ for every $a$), and the \emph{left involutory law} (the unique solution to $a*x=b$ is $x=a*b$; the property is also called \emph{left symmetry}). 
The background is described e.g. in \cite{Kik-motivation}.

\medskip
\emph{Group conjugation}, $a*b=aba^{-1}$ on any subset of a group closed with respect to conjugation, is another prototypical self-distributive operation.
This observation is often attributed to Conway and Wraithe \cite{Mos}, who also coined the the term \emph{wrack of a group}, although the idea to represent self-distributive quasigroups by conjugation appeared earlier in \cite{Ste-conj} by Stein. The conjugation operation is idempotent, left distributive, but again, rarely a quasigroup: only solutions to the equation $a*x=b$ are guaranteed to exist uniquely. Algebraic structures satisfying the three conditions are called \emph{quandles} nowadays.
(The word \emph{quandle} has no meaning in English and was entirely made up by Joyce \cite{Joy}. Many other names have been introduced for quandles, such as \emph{automorphic sets}, \emph{pseudo-symmetric sets}, \emph{left distributive left quasigroups}, etc.)

In early 1980s, Joyce \cite{Joy} and Matveev \cite{Mat}, independently, picked up the idea of ``wracking a group" to extract the essential part of the fundamental group of a knot complement. Unlike the fundamental group, the resulting structure, called the \emph{fundamental quandle} of a knot, is a full invariant of (tame, oriented) knots (up to reverse mirroring) with respect to ambient isotopy. Ever since, quandles were successfully used in knot theory to design efficiently computable invariants, see e.g. \cite{Car,FLS}.

The works of Joyce and Matveev put the foundations for the modern theory of quandles, which covers, to some extent, many traditional aspects of self-distributivity as a special case (self-distributive quasigroups, or \emph{latin quandles}, in particular). It is the main purpose of the present paper to overview the classical results on self-distributive quasigroups, and relate them to the results in modern quandle theory.

\subsection{Contents of the paper}

The paper is organized as a guide to the literature on self-distributive quasigroups, or latin quandles, trying to relate the results of various mathematical schools, which are often fairly hard to find and navigate (at least to me, due to a combination of writing style, terminology mess, and, to most mathematicians, language barrier). 

As in most survey tasks, I had to narrow down my focus. 
The main subject of the paper are representation theorems, serving as the main tool to investigate self-distributive algebraic structures, such as quandles and quasigroups. To see the tools in action, my subjective choice are enumeration results. Other interesting results are cited and commented. I do not claim completeness of my survey, and apologize in advance for eventual ignorance.

\medskip
In Section \ref{sec:back}, we overview the background from the theory of quasigroups, loops and from universal algebra. 
First, we recall various equational properties of quasigroups and quandles, and define the multiplication groups.
Then, various weakenings of the associative and commutative laws are introduced, with a focus towards the classes of commutative Moufang loops and Bruck loops, which are used in the representation theorems.
Finally, we talk about isotopy, linear and affine representation, and polynomial equivalence between quasigroups and loops.

Section \ref{sec:dq} addresses distributive and trimedial quasigroups. In the first part, we prove the classical affine representation of medial quasigroups (Theorem \ref{thm:medial}), outline Kepka's affine representation of trimedial quasigroups over commutative Moufang loops (Theorem \ref{thm:trimedial}), and comment upon some special cases and generalizations. Then, in the second part, we present a few consequences of the representation theorem, namely, a classification theorem (Theorem \ref{thm:dq-structure}), enumeration results (Table \ref{t:dq}), and we also mention the property called symmetry-by-mediality.

In a short intermezzo, Section \ref{sec:ldq-conj}, we briefly comment on the Cayley-like representation of quandles using conjugation in symmetric groups, and on the construction called the core of a loop. These were some of the first families of examples of left distributive quasigroups which are not right distributive.

In Section \ref{sec:ldq-isotopy}, we investigate loop isotopes of left distributive quasigroups, so called Belousov-Onoi loops. First, we prove a representation theorem (Theorem \ref{thm:ldq}, based on more detailed Propositions \ref{p:ldq1} and \ref{p:ldq2}), and then continue with the properties of Belousov-Onoi loops (among others, Propositions \ref{p:gbl-bl}, \ref{p:gbl-cml}, \ref{p:gbl3} and Theorem \ref{t:gbruck}). We explain why, at the moment, the correspondence is of limited value for the general theory of left distributive quasigroups. Nevertheless, one special case is important: involutory left distributive quasigroups correspond to the well established class of B-loops (Theorem \ref{thm:lsldq}). The representation theorems are outlined in Figure \ref{f:outline}.

In Section \ref{sec:ldq-homog}, we introduce the homogeneous representation of connected quandles, which is perhaps the strongest tool to study self-distributive quasigroups developed so far. We present several applications to the structure theory, with most attention paid to enumeration results.
 
\begin{figure}
\begin{center}
\begin{tikzpicture}[description/.style={fill=white,inner sep=2pt}]
\matrix (m) [matrix of math nodes, row sep=2em, column sep=1em, text height=1.5ex, text depth=0.25ex]
{ 
\text{\bf quasigroups}\ & & \ \text{\bf loops} &\\
%(Q,*) & & x\cdot y=(x/e)*(e\ld y) &\\
%&&&\\
\text{medial}\quad & & \ \text{abelian groups} &\\
\text{distributive (trimedial)}\ & & \ \text{commutative Moufang loops} & \\
& \text{involutory l.d.}\ & & \ \text{B-loops} \\
\text{left distributive}\ & & \ \text{Belousov-Onoi loops}& \\
};
\path[<->,font=\scriptsize]
%(m-1-1) edge node[auto] { } (m-1-3) 
(m-2-1) edge node[auto] {Theorem \ref{thm:medial}} (m-2-3) 
(m-3-1) edge node[auto] {Theorems \ref{thm:trimedial} and \ref{thm:dq}} (m-3-3)
(m-5-1) edge node[auto] {Theorem \ref{thm:ldq}} (m-5-3)
(m-4-2) edge node[auto] {Theorem \ref{thm:lsldq}} (m-4-4); 
\path[-,font=\scriptsize]
(m-2-1) edge node[auto] { } (m-3-1) 
(m-3-1) edge node[auto] { } (m-5-1) 
(m-5-1) edge node[auto] { } (m-4-2) 
(m-2-3) edge node[auto] { } (m-3-3) 
(m-3-3) edge node[auto] { } (m-5-3) 
(m-5-3) edge node[auto] { } (m-4-4) ;
%\path[|->,font=\scriptsize]
%(m-2-1) edge node[auto] { } (m-2-3) ;
\end{tikzpicture}
\end{center}
\caption{Correspondence between certain classes of quasigroups and loops.}
\label{f:outline}
\end{figure}
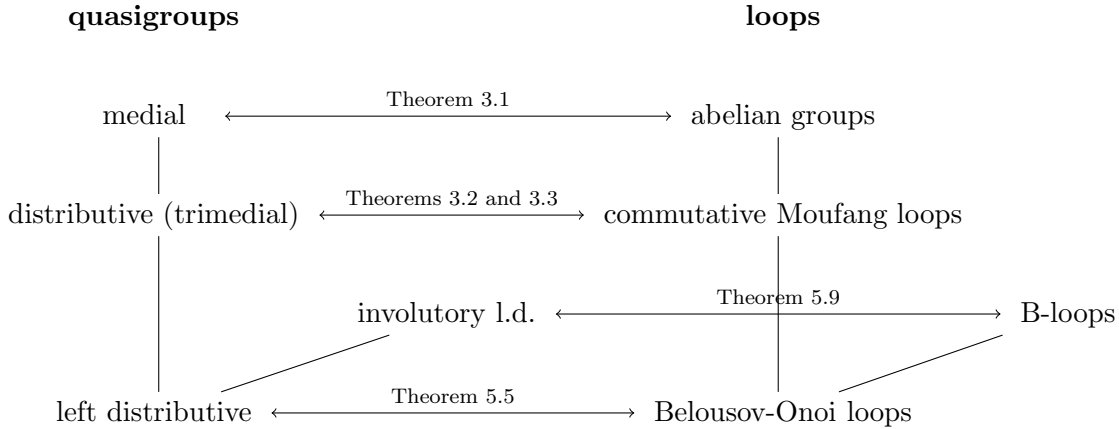

\medskip
Many proofs in our paper are only referenced.
In the case of trimedial and distributive quasigroups (Theorems \ref{thm:trimedial} and \ref{thm:dq}), we believe that new, shorter, and conceptually cleaner proofs are possible, using modern methods of universal algebra, but we did not succeed to make a substantial progress yet. The only minor contribution in this part is yet another proof of the Toyoda-Murdoch-Bruck theorem on medial quasigroups (Theorem \ref{thm:medial}).
Neither we go into details in Section \ref{sec:ldq-homog} on homogeneous representation, since it has been presented in our recent paper \cite{HSV}. On the other hand, many details are given in Section \ref{sec:ldq-isotopy}, the Belousov-Onoi theory is presented in a substantially different way. In particular, we provide a new and cleaner proof of the representation theorem for left distributive quasigroups (Theorem \ref{thm:ldq}), which contains as a special case the classical results of Belousov on distributive quasigroups (a part of Theorem \ref{thm:dq}), and the Kikkawa-Robinson theorem on involutory left distributive quasigroups (Theorem \ref{thm:lsldq}).

%(A different approach to homogeneous representation of self-distributive quasigroups is developed in \cite{Smi-qq}.)
%In \cite{Gal-smooth}, the idea of homogeneous representation is extended towards smooth quasigroups.

\subsection{A remark on automated theorem proving} 

Many theorems discussed in the present paper admit a short first order theory formulation, and subsequently could be attempted by automated theorem proving (ATP). Most of them are beyond the capabilites of current provers, but a few can be proved by any state-of-the art theorem prover within a few seconds. In those cases, we do not always bother to provide a reference or a proof, considering such problems ``easy symbolic manipulation", although it may be rather intricate to find a proof without the aid of a computer. We refer to \cite{PS} for more information about automated theorem proving in algebra.

\section{Background}\label{sec:back}

\subsection{Quasigroups and quandles}\label{ss:qgr_qua}

Let $(A,*)$ be an algebraic structure with a single binary operation, or, shortly, a \emph{binary algebra} (also referred to as \emph{magma} or \emph{groupoid} elsewhere). 
We say it possesses \emph{unique left division}, if for every $a,b\in A$, there is a unique $x\in A$ such that $a*x=b$; such an $x$ is often denoted $x=a\ld b$.
\emph{Unique right division} is defined dually: for every $a,b\in A$, there is a unique $y\in A$ such that $y*a=b$; such a $y$ is often denoted $y=b/a$.
Binary algebras with unique left and right division are called \emph{quasigroups}. 

We list a few identities which are met frequently (all identities are assumed to be universally quantified, unless stated otherwise). A binary algebra $(A,*)$ is called
\begin{itemize}
	\item \emph{left distributive} if $x*(y*z)=(x*y)*(x*z)$,
	\item \emph{right distributive} if $(z*y)*x=(z*x)*(y*x)$,
	\item \emph{distributive} if it is both left and right distributive,
	\item \emph{medial} if $(x*y)*(u*v)=(x*u)*(y*v)$,
	\item \emph{trimedial} if every 3-generated subquasigroup is medial,
	\item \emph{idempotent} if $x*x=x$,
	\item \emph{left involutory} (or \emph{left symmetric}) if $x*(x*y)=y$ (hence we have unique left division with $x\ld y=x*y$).
\end{itemize}
Observe that left distributive quasigroups are idempotent: $x*(x*x)=(x*x)*(x*x)$ by left distributivity and we can cancel from the right. Non-idempotent medial quasigroups exist, indeed, abelian groups are examples.
Also observe that idempotent trimedial binary algebras are distributive: given $a,b,c\in A$, the subalgebra $\langle a,b,c\rangle$ is medial, hence $(a*b)*(a*c)=(a*a)*(b*c)=a*(b*c)$, and dually for right distributivity; it requires quite an effort to prove the converse for quasigroups, see Theorem \ref{thm:dq}.

A binary algebra is called a (left) \emph{quandle}, if it is idempotent, left distributive and has unique left division (remarkably, the three conditions correspond neatly to the three Reidemeister moves in knot theory, see \cite{Car,Nel}). Quandles that also have unique right division are called \emph{latin quandles}. Indeed, latin quandles and left distributive quasigroups are the very same things.

For universal algebraic considerations, it is often necessary to consider quandles as algebraic structures with two binary operations, $(A,*,\ld)$, and   quasigroups as structures with three binary operations, $(A,*,/,\ld)$. Then, subalgebras are really quandles (quasigroups, respectively), etc. We will implicitly assume the division operations to be part of the algebraic structure whenever needed (e.g. when considering term operations in Section \ref{ss:linear}).

\medskip
Given a binary algebra $(A,*)$, it is natural to consider \emph{left translations} $L_a(x)=a*x$, and \emph{right translations} $R_a(x)=x*a$, and the semigroups they generate, the \emph{left multiplication semigroup} $\lmlt(A,*)=\langle L_a:a\in A\rangle$, the \emph{right multiplication semigroup} $\rmlt(A,*)=\langle R_a:a\in A\rangle$, and the \emph{multiplication semigroup} $\mlt(A,*)=\langle L_a,R_a:\ a\in A\rangle$. Unique left division turns left translations into permutations, and thus the left multiplication semigroup into a group (and dually for right translations). Observe that $L_a^{-1}(x)=a\ld x$ and $R_a^{-1}(x)=x/a$. Also note that
$(A,*)$ is left distributive if and only if $L_a$ is an endomorphism for every $a\in A$. Hence, in quandles, $\lmlt(A,*)$ is a subgroup of the automorphism group.

A binary algebra $(A,*)$ is called \emph{homogeneous} if $\aut(A,*)$ acts transitively on $A$. It is called \emph{left connected} if $\lmlt(A,*)$ acts transitively on $A$ (we will omit the adjective ``left" for quandles). A finite quandle is therefore connected if, for every $a,b\in A$, there exist $x_1,\dots,x_n\in A$ such that $b=x_1*(x_2*(\ldots(x_n*a)))$ (compare to unique right division!).
Connected quandles are arguably the most important class of quandles, both from the algebraic and topological points of view. Indeed, latin quandles are connected, and the class of connected quandles is a very natural generalization of left distributive quasigroups: many structural properties of left distributive quasigroups extend to connected quandles, as we shall see throughout Section~\ref{sec:ldq-homog}. 

To illustrate the power of connectedness, let us prove the following implication for quandles that are (both left and right) distributive.

\begin{proposition}[{\cite[Theorem 5.10]{CEHSY}}]
Finite connected distributive quandles are quasigroups.
\end{proposition}

\begin{proof}
Assume the contrary, and let $(Q,*)$ be the smallest counterexample. Right distributivity says that every right translation $R_a$ is a homomorphism, hence, its image, $R_a(Q)$, forms a subquandle that is also connected and distributive (both properties project to homomorphic images). For every $a,b\in Q$, the subquandles $R_a(Q)$ and $R_b(Q)$ are isomorphic: connectedness of $(Q,*)$ provides an automorphism $\alpha\in\lmlt(Q,*)$ such that $\alpha(a)=b$, and it follows from $\alpha(x*a)=\alpha(x)*\alpha(a)=\alpha(x)*b$ that $\alpha$ restricts to an isomorphism between $R_a(Q)$ and $R_b(Q)$.  Therefore, by minimality, all subquandles $R_a(Q)$ are proper subquasigroups. Now we prove that $R_a(Q)\subseteq R_{x*a}(Q)$ for every $x,a\in Q$. Let $y*a\in R_a(Q)$. Since $R_a(Q)$ is a quasigroup, there is $z*a\in R_a(Q)$ such that $y*a=(z*a)*(x*a)$. Hence $y*a\in R_{x*a}(Q)$. By induction, $R_a(Q)\subseteq R_{x_1*a}(Q)\subseteq R_{x_2*(x_1*a)}(Q)\subseteq\ldots$, and thus, from connectedness, $R_a(Q)\subseteq R_b(Q)$ for every $a,b\in Q$. Hence all subquasigroups $R_a(Q)$ are equal, and since $x\in R_x(Q)$ for every $x\in Q$, all of them are equal to $Q$, a contradiction.
\end{proof}

\subsection{Loops}\label{ssec:loops}

A \emph{loop} is a quasigroup $(Q,\cdot)$ with a \emph{unit} element $1$, i.e. $1\cdot a=a\cdot 1=a$ for every $a\in A$.
In the present paper, loops will be denoted multiplicatively. To avoid parenthesizing, we shortcut $x\cdot yz=x\cdot(y\cdot z)$ etc., and we remove parentheses whenever the elements associate, i.e. write $xyz$ whenever we know that $x\cdot yz=xy\cdot z$. For all unproved statements, we refer to any introductory book on loops, such as \cite{Bru,Pfl-book}.

Let $(Q,\cdot)$ be a loop. \emph{Inner mappings} are those elements of the multiplication group $\mlt(Q,\cdot)$ that fix the unit element. For example, the conjugation mappings $T_x(z)=xz/x$ are inner and, in a way, measure the non-commutativity in the loop. The \emph{left inner mappings} are defined by $L_{x,y}(z)=(xy)\ld(x\cdot yz)$ and measure the non-associativity from the left.

The most common example of loops are groups (i.e. associative loops), and most classes of loops studied in literature are those satisfying a weak version of associativity or commutativity. We list a few weak associative laws (note that all the conditions hold in groups): a loop is called
\begin{itemize}
	\item \emph{diassociative} if all 2-generated subloops are associative;
	\item \emph{left alternative} if $x\cdot xy=x^2y$;
	\item \emph{power-associative} if all 1-generated subloops are associative;
	\item \emph{Moufang} if $(xy\cdot x)z=x(y\cdot xz)$ (the dual law is equivalent in loops);
	\item \emph{left Bol} if $(x\cdot yx)z=x(y\cdot xz)$;
	\item \emph{automorphic} if all inner mappings are automorphisms.
	\item \emph{left automorphic} if all left inner mappings $L_{x,y}$ are automorphisms.
\end{itemize}
Moufang's theorem \cite{Dra} says that in a Moufang loop, every subloop generated by three elements that associate, is associative. In particular, Moufang loops are diassociative, since $a(ba)=(ab)a$ for every $a,b$, as directly follows from the Moufang law. 
Bol loops are power-associative.
%A loop is Moufang if and only if it is left Bol and diassociative, if and only if it is both left and right Bol. 

The \emph{nucleus} of a loop $(Q,\cdot)$ is the set of all elements $a\in Q$ that associate with all other elements, i.e. 
\[N=\{a\in Q:\ a\cdot xy=ax\cdot y,\ x\cdot ay=xa\cdot y,\ x\cdot ya=xy\cdot a\text{ for all }x,y\in Q\}.\]
An element of a loop is called \emph{nuclear} if it belongs to the nucleus.
A mapping $f:Q\to Q$ is called \emph{$k$-nuclear} if $x^k f(x)\in N$ for every $x\in Q$.
%For example, in loops with the inverse property, the mapping $f(x)=x^{-1}$ is 1-nuclear.
%Note that the definition of $k$-nuclearity makes much more sense in loops where the nucleus is a normal subloop, for example, in Moufang loops.

\medskip
Commutative Moufang loops were a central topic in the Bruck's book \cite{Bru}, and newer results are surveyed in \cite{Ben-book,Smi-cml}.
The following characterization shows how natural the class is.

\begin{theorem}[\cite{Bru,Pfl-cml}]\label{t:cml}
The following are equivalent for a commutative loop $(Q,\cdot)$:
\begin{enumerate}
	\item it is diassociative and automorphic;
	\item it is Moufang;
	\item[(2')] the identity $xx\cdot yz=xy\cdot xz$ holds.
	\item[(3)] the identity $f(x)x\cdot yz=f(x)y\cdot xz$ holds for some $f:Q\to Q$.
\end{enumerate}
Moreover, if $(Q,\cdot)$ is a commutative Moufang loop, than the identity of (3) holds if and only if $f$ is a $(-1)$-nuclear mapping.
\end{theorem}

The equivalence of (1), (2), (2') is well-known \cite{Bru}. The rest is a special case of a lesser known, but intriguing characterization of Moufang loops by Pflugfelder \cite{Pfl-cml}. It is one of the crucial ingrediences in Kepka's proof of Theorem \ref{thm:trimedial}, and also in our new proof of Proposition \ref{p:gbl-cml}.

\begin{example}\label{e:cml81}
According to Kepka and N\v emec \cite[Theorem 9.2]{KN}, the smallest non-associative commutative Moufang loops have order 81, there are two of them (up to isomorphism), and can be constructed as follows. Consider the 
groups $G_1=(\Z_3)^4$ and $G_2=(\Z_3)^2\times\Z_9$. Let $e_1,e_2,e_3(,e_4)$ be the canonical generators. 
Let $t_1$ be the triaditive mapping over $G_1$ satisfying 
\[t_1(e_2,e_3,e_4)=e_1,\ t_1(e_3,e_2,e_4)=-e_1,\ t_1(e_i,e_j,e_k)=0\text{ otherwise.}\]
Let $t_2$ be the triaditive mapping over $G_2$ satisfying 
\[t_2(e_1,e_2,e_3)=3e_3,\ t_2(e_2,e_1,e_3)=-3e_3,\ t_2(e_i,e_j,e_k)=0\text{ otherwise.}\]
The loops $Q_i=(G_i,\cdot)$, $i=1,2$, with \[x\cdot y=x+y+t_i(x,y,x-y),\]
are non-isomorphic commutative Moufang loops, and every commutative Moufang loop of order 81 is isomorphic to one of them. 
\end{example}

In an arbitrary loop $(Q,\cdot)$, we can define the left inverse as $x^{-1}=x\ld 1$ (in general, $x\ld 1\neq1/x$).
Then, the \emph{left inverse property} (LIP) requests that $a\backslash b=a^{-1}b$ for every $a,b\in Q$, and 
the \emph{left automorphic inverse property} (LAIP) requests that $(ab)^{-1}=a^{-1}b^{-1}$ for every $a,b\in Q$.
The RIP and RAIP are defined dually; if left and right inverses coincide, we talk about IP and AIP.

Diassociative loops have the IP, and then, commutativity is indeed equivalent to the AIP. 
Bol loops have the LIP, and are power associative, hence the left and right inverses coincide. 
Occasionally, we will need the following technical lemma.

\begin{lemma}[\cite{Kie} or ATP]\label{l:bruck}
The following properties are equivalent for a left Bol loop $(Q,\cdot)$:
\begin{enumerate}
	\item the AIP;
	\item the identity $(xy)^2=x\cdot y^2x$;
	\item $L_{ab}^2=L_aL_b^2L_a$ for every $a,b\in Q$.
\end{enumerate}
\end{lemma}

It seems that the AIP is the appropriate generalization of commutativity into the Bol setting (commutativity is no good, as it implies the Moufang law). We have the following ``left version" of Theorem \ref{t:cml}, under the additional assumption of \emph{unique 2-divisibility}, which states that the mapping $x\mapsto x^2$ is a permutation.

\begin{theorem}[\cite{Kik2} and ATP]\label{t:bruck}
The following are equivalent for a uniquely 2-divisible loop $(Q,\cdot)$ with the LAIP:
\begin{enumerate}
	\item[(1)] it has the LIP, is left alternative and left automorphic;
	\item[(1')] the identities $x^2\cdot x^{-1}y=xy$ and $L_{x,y}(z^{-1})=L_{x,y}(z)^{-1}$ hold;
	\item[(2)] it is left Bol;
	\item[(2')] the identity $(xy)^2\cdot(x^{-1}z)=x\cdot y^2z$ holds.
\end{enumerate}
\end{theorem}

\begin{proof}[Proof sketch]
(1') is an immediate consequence of (1), and (2') easily follows from (2) by Lemma \ref{l:bruck}, but the converse implications are trickier; we could not find them anywhere in literature, but they can be verified by an automated theorem prover.

To prove that the equivalent conditions (1),(1') are in turn equivalent to the equivalent conditions (2),(2'), we can use \cite[Theorem 3]{Kik2}, which states that, for left alternative uniquely 2-divisible loops with the LIP and LAIP, the identity (2') is equivalent to being left automorphic. 
% An indirect proof via BO, Kikkawa, BF. It is a special case of a more general Theorem \cite{t:gbl}, via Lemma \ref{l:???}.
\end{proof}
 
Left Bol loops with the AIP are called \emph{Bruck loops} (or \emph{K-loops} or \emph{gyrocommutative gyrogroups}). A lot of structure theory is collected in Kiechle's book \cite{Kie}. Uniquely 2-divisible Bruck loops were called \emph{B-loops} (we will use the shortcut, too) and studied in detail by Glauberman \cite{Gla}. A finite Bruck loop is uniquely 2-divisible if and only if it has odd order \cite[Proposition 1]{Gla}. Every B-loop can be realized as a subset $Q$ of a group $(G,\circ)$ such that the mapping $x\mapsto x\circ x$ is a permutation on $Q$ and the loop operation is $a\cdot b=\sqrt a\circ b\circ\sqrt a$ \cite[Theorem 2]{Gla}.

\begin{example}\label{e:bloop15}
The smallest non-associative B-loop has order 15 and can be constructed as follows. Consider the loop $(\Z_5\times\Z_3,\cdot)$ with \[(a,x)\cdot(b,y)=(\varphi_{x,y}a+b,x+y)\]
where $\varphi_{x,y}\in\Z_5^*$ are given by the following table:
\begin{displaymath}
    \begin{array}{c|ccc}
          & 0 & 1 & 2 \\\hline
        0 & 1 & 2 & 2 \\
        1 & 1 & 3 & 1 \\
        2 & 1 & 1 & 3 \\
    \end{array}
\end{displaymath}
It is straightforward to check that this is a B-loop. It is an abelian extension of $\Z_5$ by $\Z_3$ in the sense of \cite{SV}.
\end{example}
 
\subsection{Linear and affine representation}\label{ss:linear}

A great portion of the present paper is about establishing that ``two algebraic structures are essentially the same". To formalize the statement, we borrow a formal definition from universal algebra. Let $(A,f_1,f_2,\dots)$ be an arbitrary algebraic structure (shortly, algebra), with basic operations $f_1,f_2,\ldots$ A \emph{term operation} is any operation that results as a composition of the basic operations. \emph{Polynomial operations} result from term operations by substituting constants for some of the variables. Two algebras with the same underlying set are called \emph{term equivalent} (or \emph{polynomially equivalent}, respectively), if they have the same term operations (or polynomial operations). For example, a group can be presented in the standard way, as $(G,\cdot,^{-1},1)$, or in the loop theoretical way, as an associative loop $(G,\cdot,/,\ld,1)$; the two algebraic structures are formally different, but they are term equivalent, since the basic operations in any one of them are term operations in the other one. Term equivalent algebras have identical subalgebras, polynomially equivalent algebras have identical congruences, and share all properties that only depend on terms or polynomials (for example, the Lagrange property, see Section \ref{ss:ldq-structure}). To learn more, consult \cite[Section 4.8]{Ber}.

\medskip
One of the fundamental tools to study a quasigroup is, to determine its loop isotopes, and use the properties of the loops to obtain an information about the original quasigroup. An \emph{isotopy} between two quasigroups $(Q_1,*)$ and $(Q_2,\cdot)$ is a triple of bijective mappings
$\alpha,\beta,\gamma:Q_1\to Q_2$ such that \[\alpha(a)\cdot\beta(b)=\gamma(a*b)\] for every $a,b\in Q_1$. Then, $(Q_2,\cdot)$ is called an \emph{isotope} of $(Q_1,*)$. The combinatorial interpretation is that $(Q_2,\cdot)$ is obtained from $(Q_1,*)$ by permuting rows, columns and renaming entries in the multiplication table. Up to isomorphism, we can only consider isotopes with $Q_1=Q_2$ and $\gamma=id$, so called \emph{principal isotopes}.

Every quasigroup admits many principal loop isotopes, often falling into more isomorphism classes, yet all of them have a particularly nice form.

\begin{proposition}[{\cite[Section III]{Bru}}]
Let $(Q,*)$ be a quasigroup and $\alpha,\beta$ permutations on $Q$. The following are equivalent:
\begin{itemize}
	\item the isotope $a\cdot b=\alpha(a)*\beta(b)$ is a loop;
	\item $\alpha=R_{e_1}$ and $\beta=L_{e_2}$ for some  $e_1,e_2\in Q$.
\end{itemize}
\end{proposition}

Rephrased, given a quasigroup $(Q,*)$, the only loop isotopes, up to isomorphism, are $(Q,\cdot)$ with
\[a\cdot b=(a/e_1)*(e_2\ld b),\]
where $e_1,e_2\in Q$ can be chosen arbitrarily. Then the unit element is $1=e_2*e_1$. For the division operations, we will use the symbols $\ldd$ and $\rdd$, to distinguish them from the quasigroup division.

Notice that the new operation $\cdot$ is a polynomial operation over the original quasigroup, and so are the division operations. 
We can recover the quasigroup operation as \[a*b=R_{e_1}(a)\cdot L_{e_2}(b),\]
but this is rarely a polynomial operation over $(Q,\cdot)$.
The most satisfactory loop isotopes are those where $R_{e_1}$ and $L_{e_2}$ are affine mappings over $(Q,\cdot)$. 

\medskip
A permutation $\varphi$ of $Q$ is called \emph{affine} over $(Q,\cdot)$, if 
\[\varphi(x)=\tilde\varphi(x)\cdot u\quad\text{or}\quad\varphi(x)=u\cdot\tilde\varphi(x)\]
where $\tilde\varphi$ is an automorphism of $(Q,\cdot)$ and $u\in Q$. In other terms, if $\varphi=R_u\tilde\varphi$ or $\varphi=L_u\tilde\varphi$. 
%Also notice that in automorphic loops, the inner mapping $T_u$ is an automorphism, hence $\varphi=L_u\tilde\varphi=R_uR_u^{-1}L_u\tilde\varphi=R_u(T_u\tilde\varphi)$ and the two forms are equivalent.
A quasigroup $(Q,*)$ is called \emph{affine} over a loop $(Q,\cdot)$ if, for every $a,b\in Q$, \[a*b=\varphi(a)\cdot\psi(b),\] where $\varphi,\psi$ are affine mappings over $(Q,\cdot)$ such that $\tilde\varphi\tilde\psi=\tilde\psi\tilde\varphi$. If both $\varphi,\psi$ are automorphisms, we call $(Q,*)$ \emph{linear} over $(Q,\cdot)$. (Note that the affine mappings $\varphi,\psi$ do not necessarily commute.)

\begin{example}\label{e:medial_is_affine}
To illustrate the concept of affine representation, consider a quasigroup $(Q,*)$ affine over an abelian group $(Q,\cdot)$. We prove that it is medial. 
With $\varphi=R_u\tilde\varphi$, $\psi=R_v\tilde\psi$ (left or right makes no difference here), we have
\begin{align*}
(a*b)*(c*d)&=\varphi\left(\varphi(a)\cdot\psi(b)\right)\cdot\psi\left(\varphi(c)\cdot\psi(d)\right)\\
&=\tilde\varphi\left(\tilde\varphi(a)u\cdot\tilde\psi(b)v\right)u\cdot\tilde\psi\left(\tilde\varphi(c)u\cdot\tilde\psi(d)v\right)v\\
&=\tilde\varphi^2(a)\cdot\tilde\varphi\tilde\psi(b)\cdot\tilde\psi\tilde\varphi(c)\cdot\tilde\psi^2(d)\cdot\tilde\varphi(uv)\cdot\tilde\psi(uv)\cdot uv.
\end{align*}
Since $\tilde\varphi\tilde\psi=\tilde\psi\tilde\varphi$, the expression is invariant with respect to interchange of $b$ and $c$.
As we shall see, Theorem \ref{thm:medial} states also the converse: every medial quasigroup is affine over an abelian group.
\end{example}

Any adjective to the words ``affine" or ``linear" will refer to the properties of the mappings $\varphi$ and $\psi$. In Section \ref{sec:dq}, we will consider 1-nuclear affine representations over commutative Moufang loops, i.e. we will assume that $\varphi,\psi$ are 1-nuclear affine mappings.
Notice that if $\varphi=F_u\tilde\varphi$, with $F\in\{L,R\}$, is 1-nuclear, then $u$ is nuclear (substitute $1$), and if the nucleus is a normal subloop, then $\tilde\varphi$ is also 1-nuclear. 

How to turn an affine representation into a polynomial equivalence? Consider affine mappings $\varphi=F_u\tilde\varphi$, $\psi=G_v\tilde\psi$ where $F,G\in\{L,R\}$ and $\tilde\varphi,\tilde\psi$ are automorphisms of $(Q,\cdot)$. Then $x*y=\varphi(x)\cdot\psi(y)$ is a polynomial operation over the algebra $(Q,\cdot,\tilde\varphi,\tilde\psi)$, and a similar statement applies to the division operations, too (one also needs to use the inverse automorphisms $\tilde\varphi^{-1},\tilde\psi^{-1}$). Conversely, if $(Q,\cdot)$ is a loop isotope of a quasigroup $(Q,*)$, then $x\cdot y=(x/e_1)*(e_2\ld y)$, $x\ldd y=e_2*((x/e_1)\ld y)$, and $x\rdd y=(x/(e_2\ld y))*e_1$ are all polynomial operations over the quasigroup. If the translations $R_{e_1},L_{e_2}$ are affine over $(Q,\cdot)$, then $\tilde R_{e_1}(x)=(x*e_1)\rdd(1*e_1)$, $\tilde L_{e_2}(x)=(e_2*x)\rdd(e_2*1)$ are polynomial operations, too, hence the quasigroup $(Q,*,\ld,/)$ and the algebra $(Q,\cdot,\ldd,\rdd,\tilde R_{e_1},\tilde R_{e_1}^{-1},\tilde L_{e_2},\tilde L_{e_2}^{-1})$ are polynomially equivalent, i.e. essentially the same object.
It is convenient to perceive the loop expanded by two automorphisms in a module-theoretic way, as we shall explain now.

The classical case first: assume the loop is an abelian group and let us denote it additively, $(Q,+)$. Let $\varphi,\psi$ be two commuting automorphisms of $(Q,+)$.
Then the algebra $(Q,+,-,0,\varphi,\varphi^{-1},\psi,\psi^{-1})$ is term equivalent to the module over the ring of Laurent polynomials $\Z[s,s^{-1},t,t^{-1}]$ whose underlying additive structure is $(Q,+)$ and the action of $s,t$ is that of $\varphi,\psi$, respectively. The corresponding quasigroup operation can be written as the affine form \[x*y=sx+ty+c,\] where $c\in Q$ is a constant. 

For general loops, one can consider ``generalized modules" over commutative ``generalized rings", where the underlying additive structures are not necessarily associative. No general theory has been developed yet, but there are indications that this approach could provide a powerful tool. For example, commutative diassociative loops share a lot of module-theoretic properties of abelian groups, such as the primary decomposition \cite{KV}. The idea of ``generalized modules" and the corresponding homological methods have been exploited several times to prove interesting theorems about quasigroups \cite{HKN,Hou,KKP2}.

\medskip
Finally, let us note that our definition of affine quasigroup is too strong in one sense, and possibly weak in another sense.

The condition that the two automorphisms $\tilde\varphi,\tilde\psi$ commute is strongly tied to mediality and its weaker forms, and we included it only for brevity. Omitting the condition makes a very good sense from the universal algebra point of view. Quasigroups that admit a ``non-commuting" affine representation over an abelian group (and thus polynomially equivalent to a module over the ring of Laurent polynomials of two non-commuting variables) have been studied since the 1970s, see \cite[Chapter 3]{Smi-rep} or \cite{Dra-gps} for recent developments (the original name \emph{T-quasigroups} is slowly fading away, being replaced by the adjective \emph{central}; in universal algebra, they would be called \emph{abelian} or \emph{affine}, as the two concepts are equivalent for quasigroups).

In Section \ref{sec:dq}, all affine representations will be 1-nuclear. However, we resist to enforce nuclearity in the definition of affineness, since we do not understand its role properly (in particular, we do not know whether the representation of Theorem \ref{thm:ldq} admits any sort of nuclearity). We are not yet certain what is the appropriate generalization of the notion of an affine form into the non-associative setting.

\newpage

\section{Distributive quasigroups}\label{sec:dq}

\subsection{Affine representation}\label{ss:dq-lin}

The first ever affine representation theorem was the one for medial quasigroups, proved independently by Toyoda \cite{Toy}, Murdoch \cite{Mur} and Bruck \cite{Bru-AMS1} in the 1940s.

\begin{theorem}[\cite{Bru-AMS1,Mur,Toy}]\label{thm:medial}
The following are equivalent for a quasigroup $(Q,*)$:
\begin{enumerate}
	\item it is medial;
 	\item it is affine over an abelian group.
\end{enumerate}
\end{theorem}

\begin{proof}
$(2)\Rightarrow(1)$ was calculated in Example \ref{e:medial_is_affine}.

$(1)\Rightarrow(2)$.
Pick arbitrary $e_1,e_2\in Q$ and define a loop operation on $Q$ by $a\cdot b=(a/e_1)*(e_2\ld b)$. We can recover the quasigroup operation as $a*b=R_{e_1}(a)\cdot L_{e_2}(b)$, where $R_{e_1},L_{e_2}$ are translations in $(Q,*)$. We show that $(Q,\cdot)$ is an abelian group, and that $R_{e_1},L_{e_2}$ are affine mappings over $(Q,\cdot)$.

First, consider the quasigroup $(Q,\circ)$ with $a\circ b=(a/e_1)*b$. We prove that it is also medial. 
Observe that, for every $x,y,u,v\in Q$, 
\[(x/y)*(u/v)=(x*u)/(y*v),\tag{\dag}\] 
since $((x/y)*(u/v))*(y*v)=((x/y)*y)*((u/v)*v)=x*u$, and we obtain the identity by division from the right. 
Now we expand
\begin{align*}
(a\circ b)\circ(c\circ d)
&=(((a/e_1)*b)/e_1)*((c/e_1)*d)\\
&=(((a/e_1)*b)/((e_1/e_1)*e_1))*((c/e_1)*d)\\
&=(((a/e_1)/(e_1/e_1))*(b/e_1))*((c/e_1)*d),
\end{align*}
and using mediality, we can interchange $b/e_1$ and $c/e_1$, and by an analogous calculation obtain $(a\circ b)\circ(c\circ d)=(a\circ c)\circ(b\circ d)$.
Now notice that $a\cdot b=a\circ(e_2\ld b)=a\circ((e_2*e_1)\ld^\circ b)$, hence a dual argument, with $*$ replaced for $\circ$ and $e_1$ replaced for $e_2*e_1$, shows that the loop $(Q,\cdot)$ is also medial. But medial loops are abelian groups.

It remains to prove that the mappings $R_{e_1},L_{e_2}$ are affine over $(Q,\cdot)$ and that the corresponding automorphisms $\tilde R_{e_1},\tilde L_{e_2}$ commute. Let $1$ denote the unit and $^{-1}$ the inverse element in the group $(Q,\cdot)$. Consider $a,b\in Q$. By mediality, 
\[ (R_{e_1}^{-1}(a)*L_{e_2}^{-1}(b)) * (L_{e_2}^{-1}(1)*L_{e_2}^{-1}(1)) = (R_{e_1}^{-1}(a)*L_{e_2}^{-1}(1)) * (L_{e_2}^{-1}(b)*L_{e_2}^{-1}(1)).\] 
Rewriting $x*y=R_{e_1}(x)\cdot L_{e_2}(y)$, we obtain
\[ R_{e_1}(a\cdot b)\cdot L_{e_2}R_{e_1}L_{e_2}^{-1}(1) = R_{e_1}(a)\cdot L_{e_2}R_{e_1}L_{e_2}^{-1}(b).\]
With $a=1$, we obtain $L_{e_2}R_{e_1}L_{e_2}^{-1}(b)=R_{e_1}(b)\cdot L_{e_2}R_{e_1}L_{e_2}^{-1}(1) \cdot R_{e_1}(1)^{-1}$, and after replacement of the last term in the previous identity, and after cancelling the term $L_{e_2}R_{e_1}L_{e_2}^{-1}(1)$, we obtain
\[ R_{e_1}(a\cdot b) = R_{e_1}(a)\cdot R_{e_1}(b)\cdot R_{e_1}(1)^{-1}.\]
This shows that $R_{e_1}$ is an affine mapping, with the underlying automorphism $\tilde R_{e_1}(x) = R_{e_1}(x)R_{e_1}(1)^{-1}$.
Dually, we obtain that $L_{e_2}$ is an affine mapping, with the underlying automorphism $\tilde L_{e_2}(x) = L_{e_2}(x)L_{e_2}(1)^{-1}$.

Finally we show that the two automorphisms commute. With $\varphi=R_{e_1}$, $\psi=L_{e_2}$, $u=R_{e_1}(1)^{-1}$ and $v=L_{e_2}(1)^{-1}$, we can calculate as in Example \ref{e:medial_is_affine} that, for every $x\in Q$,
\[\tilde\varphi\tilde\psi(x)\cdot\tilde\varphi(uv)\cdot\tilde\psi(uv)\cdot uv
=(1*x)*(1*1)=(1*1)*(x*1)=
\tilde\psi\tilde\varphi(x)\cdot\tilde\varphi(uv)\cdot\tilde\psi(uv)\cdot uv.\]
After cancellation, we see that $\tilde\varphi\tilde\psi=\tilde\psi\tilde\varphi$.
\end{proof}

Note that we proved a stronger statement: \emph{any} loop isotope of a medial quasigroup is an abelian group that provides an affine representation. For other classes, in order to obtain an affine representation over a nice class of loops, one often has to choose the parameters $e_1,e_2$ in a special way. For instance, for trimedial quasigroups, one has to take $e_1=e_2$ which is a square, as we shall see.

Perhaps the best way to perceive distributive quasigroups is through \emph{trimediality}. As we shall see, a quasigroup is distributive if and only if it is idempotent and trimedial. This was first realized by Belousov in \cite{Bel-dq}, and his proof was based on finding an isotopy of a distributive quasigroup to a commutative Moufang loop, and subsequently using Moufang's theorem (see also his book \cite[Theorems 8.1 and 8.6]{Bel}). Belousov's method actually provides a linear representation, but this fact was recognized and explicitly formulated only later by Soublin \cite[Section II.7, Theorem 1]{Sou}. 
An analogous theorem for general (not necessarily idempotent) trimedial quasigroups was proved by Kepka \cite{Kep-trimedial} a few years later (Theorem \ref{thm:trimedial}). We will now outline Kepka's proof, and show how the Belousov-Soublin theorem follows as a special case (Theorem \ref{thm:dq}). 

Many equivalent conditions charecterizing trimediality are formulated in \cite{Kep-trimedial}, we only pick the most important ones here: (1) trimediality, (2) a stronger fact stating that mediating elements generate a medial subquasigroup, (3) a finite equational base for trimediality, and (4) the affine representation.
In fact, Kepka lists several finite bases, but not the one we state here: our condition (3) is a \emph{minimal} base, found in \cite{KP}, and subsumes most of Kepka's bases.

\begin{theorem}[\cite{Kep-trimedial}]\label{thm:trimedial}
The following are equivalent for a quasigroup $(Q,*)$:
\begin{enumerate}
	\item it is trimedial;
	\item for every $a,b,c,d\in Q$, if $(a*b)*(c*d)=(a*c)*(b*d)$ then the subquasigroup $\langle a,b,c,d\rangle$ is medial;
	\item it satisfies, for every $a,b,c\in Q$, the identities 
\begin{gather*} (c*b)*(a*a)=(c*a)*(b*a), \\(a*(a*a))*(b*c)=(a*b)*((a*a)*c); \end{gather*}
	\item it is 1-nuclear affine over a commutative Moufang loop.
\end{enumerate}
\end{theorem}

\begin{proof}[Proof sketch]
$(2)\Rightarrow(1)$. For any $a,b,c\in Q$, we have $(b*a)*(a*c)=(b*a)*(a*c)$. Hence, by (2), $\langle a,b,c\rangle$ is medial.

$(1)\Rightarrow(3)$. Given $a,b,c\in Q$, consider the subquasigroup $\langle a,b,c\rangle$. It is medial, hence the two identities hold for $a,b,c$. 

$(3)\Rightarrow(4)$. First of all, we need to prove the following two additional identities: $(a*a)*(b*c)=(a*b)*(a*c)$ and $(a*b)*(c*a)=(a*c)*(b*a)$ (in Kepka's terminology, to prove that $(Q,*)$ is a WAD-quasigroup). A proof can be found quickly by an automated theorem prover, or read in \cite{KP}.
Now we can follow Kepka's proof from \cite{Kep-trimedial}, whose structure is similar to our proof of Theorem \ref{thm:medial}.

Pick an arbitrary square $e\in Q$ (i.e. $e=e'*e'$ for some $e'$) and define the loop operation on $Q$ by $a\cdot b=(a/e)*(e\ld b)$. We can recover the quasigroup operation as $a*b=R_e(a)\cdot L_e(b)$, where $L_e,R_e$ are translations in $(Q,*)$. To show that $(Q,\cdot)$ is a commutative Moufang loop, it is sufficient to verify condition (3) of Theorem \ref{t:cml} with $f=R_eL_e^{-1}$. The proof is rather technical, see \cite[Proposition 4.8(iii)]{Kep-trimedial1}. It also follows that the mapping $f$ is (-1)-nuclear, and another technical calculation, as in \cite[Lemma 3(iii)]{Kep-trimedial}, shows that the mappings $L_e,R_e$ are 1-nuclear. Finally, we can reuse the second part of our proof of Theorem \ref{thm:medial} to show that the two mappings are affine and that the underlying automorphisms commute, since we only used the identity $(a*a)*(b*c)=(a*b)*(a*c)$ and its dual in the proof. We have to be careful about non-associativity of the multiplication, but fortunately, all calculations are correct thanks to the fact that the mappings $L_e,R_e$ are 1-nuclear, hence preserve the nucleus (in particular, all elements resulting by application of $L_e,R_e$ on 1 are nuclear).

%(Note that the first part of our proof of Theorem \ref{thm:medial} could be used to show that $(Q,\cdot)$ is commutative and diassociative, but the Moufang law requires a different approach.)

$(4)\Rightarrow(2)$. The idea is, find a subloop $Q'$ of $(Q,\cdot)$ that contains all four elements $a,b,c,d$ and is generated by three elements $u,v,w$ that associate. Then, by Moufang's theorem \cite{Dra}, $Q'$ is an abelian group, and thus the subquasigroup $\langle a,b,c,d\rangle$ is medial by Theorem \ref{thm:medial}. The construction is described in \cite[Theorem 2 (vi)$\Rightarrow$(vii)]{Kep-trimedial}.
\end{proof}

As a corollary to Theorem \ref{thm:trimedial}, we settle the case of distributive quasigroups.

\begin{theorem}[\cite{Sou}]\label{thm:dq}
The following are equivalent for an idempotent quasigroup $(Q,*)$:
\begin{enumerate}
	\item it is trimedial;
	\item for every $a,b,c,d\in Q$, if $(a*b)*(c*d)=(a*c)*(b*d)$ then the subquasigroup $\langle a,b,c,d\rangle$ is medial;
	\item it is distributive;
	\item it is 1-nuclear linear over a commutative Moufang loop.
\end{enumerate}
\end{theorem}

\begin{proof}
Look at Theorem \ref{thm:trimedial}. Conditions (1) and (2) are identical. Under the assumption of idempotence, condition (3) of Theorem \ref{thm:trimedial} is equivalent to distributivity. To obtain the equivalence of the fourth conditions, we observe that an idempotent quasigroup which is 1-nuclear affine over a commutative Moufang loop $(Q,\cdot)$ is actually linear over $(Q,\cdot)$: with $\varphi=R_u\tilde\varphi$ and $\psi=R_v\tilde\psi$, thanks to nuclearity and commutativity, we have $a*b=\tilde\varphi(a)\tilde\psi(a)uv$, and since $1=1*1=\tilde\varphi(1)\tilde\psi(1)uv=uv$ we see that $a*b=\tilde\varphi(a)\tilde\psi(a)$ is a linear representation.
\end{proof}

For idempotent quasigroups, the linear representation $a*b=\varphi(a)\cdot\psi(b)$ is determined by either one of the automorphisms $\varphi$ or $\psi$, since $a=a*a=\varphi(a)\cdot\psi(a)$, hence $\varphi(a)=a\rdd\psi(a)$ or $\psi(a)=\varphi(a)\ldd a$. Mappings $\varphi,\psi$ satisfying $\varphi(a)\cdot\psi(a)=a$ will be called \emph{companions}. Note that the companion of an automorphism is not necessarily a permutation or an endomorphism! However, if it is an endomorphism, then the two mappings commute.

\begin{example}\label{e:dq81}
Combining Theorem \ref{thm:dq} and Example \ref{e:cml81}, one can determine the smallest non-medial distributive quasigroups. They have order 81 and there are six of them (up to isomorphism) \cite[Theorem 12.4]{KN}. A careful analysis of the automorphisms of the loops $(G_1,\cdot)$ and $(G_2,\cdot)$ of Example \ref{e:cml81} (see \cite[Sections 5 and 6]{KN}, respectively) leads to the following classification:
\begin{enumerate}
	\item $(G_1,*)$ with $x*y=x^{-1}\cdot y^{-1}$.
	\item $(G_1,*)$ with $x*y=\varphi(x)\cdot\psi(y)$ where $\varphi(x)=(x_2-x_1)e_1-x_2e_2-x_3e_3-x_4e_4$ and $\psi$ is its companion.
	\item $(G_2,*)$ with $x*y=\sqrt x\cdot \sqrt y$. In $(G_2,\cdot)$, the mapping $x\mapsto x^2$ is a 1-nuclear automorphism, and so is its inverse $x\mapsto \sqrt x$.
	\item $(G_2,*)$ with $x*y=x^{-1}\cdot y^2$.
	\item $(G_2,*)$ with $x*y=x^2\cdot y^{-1}$.
	\item $(G_2,*)$ with $x*y=\varphi(x)\cdot\psi(y)$ where $\varphi(x)=-x_1e_1-x_2e_2-(3x_1+x_3)e_3$ and $\psi$ is its companion.
\end{enumerate}
\end{example}

Theorem \ref{thm:dq} has an interesting connection to design theory. It is well known that \emph{Steiner triple systems} correspond to a certain class of (finite) idempotent quasigroups, called \emph{Steiner quasigroups}. Affine Steiner triple systems, constructed over the affine spaces $(\mathbb F_3)^k$, correspond to medial Steiner quasigroups, $((\mathbb F_3)^k,*)$ with $a*b=-a-b$. \emph{Hall triple systems} can be defined by the property that every subsystem generated by three points is affine. Theorem \ref{thm:dq} implies that the corresponding quasigroups are precisely the distributive Steiner quasigroups. As a consequence, one can obtain, for instance, the enumeration of Hall triple systems, see the numbers $DQ(n)$ in Table \ref{t:dq} (the one of order 81 is item (1) of Example \ref{e:dq81}). We refer to \cite{Ben-dq,DGMOS} for details and other relations between distributive quasigroups, finite geometries and combinatorial designs.

Theorems \ref{thm:trimedial} and \ref{thm:dq} can be further generalized in several directions. 
For example, it was proved by Kepka, Kinyon and Phillips \cite[Theorem 1.2]{KKP1} that the class of \emph{F-quasigroups}, properly containing the trimedial quasigroups, admits a 1-nuclear $(-1)$-Moufang-central affine representation over \emph{NK-loops}, a class of Moufang loops that are sums of their nucleus and Moufang center.
Another direction is weakening the unique divisibility condition, see the comprehensive studies by Je\v zek, Kepka and N\v emec \cite{JK1,JK3,JKN,Kep-division,KN}. In all of these papers, a self-dual condition (such as trimediality or both-sided distributivity) is essential for linearization. The one-sided case is quite different and will be studied in Section \ref{sec:ldq-isotopy}. Nevertheless, we will be able to obtain the representation from Theorem \ref{thm:dq} as a consequence of the one-sided theory.

\subsection{Structure and enumeration}\label{ss:dq-structure}

Theorem \ref{thm:dq} allows to use the well developed theory of commutative Moufang loops to build the structure theory of distributive quasigroups.
We will describe a few examples. Further results can be found in the comprehensive survey \cite{Ben-book}.

We start with Galkin's interpretation of the Fischer-Smith theorem \cite{Gal-dq,Smi-dq}.

\begin{theorem}[\cite{Gal-dq}]\label{thm:dq-structure}
Let $Q$ be a finite distributive quasigroup of order $p_1^{n_1}\cdot\ldots\cdot p_k^{n_k}$ where $p_1,\dots,p_k$ are pairwise different primes.
Then \[Q\simeq Q_1\times \ldots\times Q_k\] where $|Q_i|=p_i^{n_i}$. Moreover, if $Q_i$ is not medial, then $p_i=3$ and $n_i\geq 4$.
\end{theorem}

The story of the proof goes as follows. Let $Q$ be a finite distributive quasigroup. The first step was Fischer's proof \cite{Fis} that $\lmlt(Q)$ is solvable, using substantial results from group theory, including the Feit-Thompson theorem and the Brauer-Suzuki theorem. Then Smith \cite{Smi-dq} was able to strengthen Fischer's theorem, while avoiding the heavy finite group machinery, by combining Theorem \ref{thm:dq} and the Bruck-Slaby theorem \cite[Chapter VIII]{Bru} stating that finite commutative Moufang loops are centrally nilpotent. Smith's result says that the derived subgroup $\lmlt(Q)'$ is the direct product of a 3-group and an abelian group of order coprime to 3 (hence $\lmlt(Q)'$ is nilpotent and $\lmlt(Q)$ is solvable, as proved by Fischer). Finally, Galkin \cite{Gal-dq} observed that his idea of minimal representation (explained in our Section \ref{sec:ldq-homog}) implies that the quasigroup $Q$ decomposes in a way analogous to the decomposition of $\lmlt(Q)'$. Using the fact that every 3-generated subquasigroup is medial (see Theorem \ref{thm:dq}), one concludes that a non-medial distributive quasigroup has at least $3^4=81$ elements. 

A somewhat different approach to the Fischer-Smith theorem, based on the homogeneous representation of Section \ref{sec:ldq-homog}, is presented in \cite{Gal-dq2}.

\medskip
An interesting story is the \emph{enumeration} of distributive quasigroups. Again, Theorem \ref{thm:dq} is crucial here, as it allows to focus on the enumeration of commutative Moufang loops and their automorphism groups. 
It is not difficult to prove (see e.g. \cite[Lemma 12.3]{KN}) that two commutative Moufang loops, $Q_1$ and $Q_2$, and their nuclear automorphisms, $\psi_1$ and $\psi_2$, respectively, provide isomorphic distributive quasigroups if and only if there is a loop isomorphism $\varphi:Q_1\to Q_2$ such that $\psi_2=\varphi\psi_1\varphi^{-1}$. 

In particular, the lemma applies to abelian groups, hence the number $MI(n)$ of medial idempotent quasigroups of order $n$ up to isomorphism can be determined using the classification of finite abelian groups and the corresponding linear algebra. The function $MI(n)$ is indeed multiplicative (i.e. $MI(mn)=MI(m)MI(n)$ for every $m,n$ coprime) and explicit formulas for $MI(p^k)$, $p$ prime and $k\leq4$, were found by Hou \cite{Hou} (in his paper, (finite) medial idempotent quasigroups are referred to as connected Alexander quandles; the formulas are given in \cite[equation (4.2)]{Hou} and the complete list of quasigroups is displayed in \cite[Table 1]{Hou}). See our Table \ref{t:ldq} for the first 47 values of $MI(n)$.

Theorem \ref{thm:dq-structure} says that the interesting (i.e. directly indecomposable) non-medial distributive quasigroups have orders $n=3^k$, $k\geq4$. Table \ref{t:dq} summarizes some of the enumeration results found in literature. $CML(n)$ denotes the number of non-associative commutative Moufang loops of order $n$ up to isomorphism, as calculated in \cite{KN}; the next four rows describe the numbers of non-medial quasigroups of order $n$ up to isomorphism in the following classes: $3M(n)$ refers to trimedial quasigroups \cite{KBL}, $D(n)$ to distributive quasigroups \cite{KN}, $DM(n)$ to distributive Mendelsohn quasigroups \cite{DGMOS}, and $DS(n)$ to distributive Steiner quasigroups \cite{Ben-dq,Kep-35}; the last row displays the medial case. 

\begin{table}[ht]
\[
\begin{array}{r|rrrrrr}
n &      3 & 3^2 & 3^3 & 3^4 & 3^5 & 3^6 \\\hline
CML(n) & 0 &   0 &   0 &   2 &   6 & \geq8  \\\hline
3M(n)  & 0 &   0 &   0 &  35 &     &     \\
D(n)   & 0 &   0 &   0 &   6 &     &     \\
DM(n)  & 0 &   0 &   0 &   2 &\geq3&     \\
DS(n)  & 0 &   0 &   0 &   1 &   1 & 3   \\\hline
MI(n)  & 1 &   8 &  30 & 166 &     & 		 \\
\end{array}
\]
\caption{Enumeration of commutative Moufang loops and of various classes of distributive quasigroups.}
\label{t:dq}
\end{table}

Another interesting enumeration result says that the smallest non-medial \emph{hamiltonian} distributive quasigroup has order $3^6$, and that there are two of them \cite{HKN}. This is perhaps the deepest application of the module-theoretical approach to distributive quasigroups.

\medskip
Finally, let us mention the property called \emph{symmetry-by-mediality}. An idempotent binary algebra is called symmetric-by-medial, if it has a congruence $\alpha$ such that its blocks are \emph{symmetric} (i.e. both left and right involutory), and the factor over $\alpha$ is medial. (In idempotent algebras, congruence blocks are always subalgebras.) Symmetric distributive quasigroups are commutative, and they are precisely the distributive Steiner quasigroups. Using Bruck's associator calculus for Moufang loops, Belousov proved that distributive quasigroups are symmetric-by-medial \cite[Theorem 8.7]{Bel}. Again, the theorem generalizes to a non-quasigroup setting \cite{JK-symbymed,Sta-symbymed}.

\section{Conjugation and cores}\label{sec:ldq-conj}

Let $(G,\cdot)$ be a group and $Q$ a subset of $G$ closed with respect to conjugation. Then the binary algebra $(Q,*)$ with \[a*b=aba^{-1}\] is a quandle, called a \emph{conjugation quandle} over the group $(G,\cdot)$. It is easy to verify that every quandle admits a \emph{Cayley-like representation} over a conjugation quandle.

\begin{proposition}
Let $(Q,*)$ be a quandle. Then $a\mapsto L_a$ is a quandle homomorphism of $(Q,*)$ onto a conjugation quandle over the group $\lmlt(Q,*)$.
\end{proposition}

\begin{proof}
Left distributivity implies $a*(b*(a\ld x))=(a*b)*x$, hence $L_a*L_b=L_aL_bL_a^{-1}=L_{a*b}$.
\end{proof}

This homomorphism is rarely an embedding, even for connected quandles. However, it is an embedding for every latin quandle, because, in a latin quandle, $L_a(x)=a*x\neq b*x=L_b(x)$ for every $a\neq b$ and every $x$. Hence, every latin quandle \emph{is} a conjugation quandle, up to isomorphism. This observation can probably be attributed to Stein \cite{Ste-conj}. He also found the following criterion.

\begin{proposition}[\cite{Ste-conj}]
Let $(G,\cdot)$ be a group, $Q$ a subset of $G$ closed with respect to conjugation, and assume that for every $a,b,c\in Q$, $aN_G(c)=bN_G(c)$ iff $a=b$. Then the conjugation quandle $(Q,*)$ is latin.
\end{proposition}

A few structural results on quandles have been proved using the Cayley representation. 
For instance, Kano, Nagao and Nobusawa \cite{KNN} used it for involutory quandles (in this case, the quandle is represented by involutions), and proved the following characterization of involutory quandles that are latin.

\begin{theorem}[\cite{KNN}]
A finite involutory quandle $(Q,*)$ is a quasigroup if and only if the derived subgroup $\lmlt(Q,*)'$ has odd order.
\end{theorem}

The proof is not easy and uses Glauberman's $Z^*$-theorem. They conclude that involutory left distributive quasigroups are solvable, and possess the Lagrange and Sylow properties (see Section \ref{ss:ldq-structure} for a more comprehensive discussion).

The Cayley representation is fundamental in Pierce's work on involutory quandles \cite{Pie}, and McCarron \cite{McC} used conjugation to represent simple quandles and to argue that there were no connected quandles with $2p$ elements, for any prime $p>5$ (see also Section \ref{ss:ldq-enumeration}).

\medskip
Let $(G,\cdot)$ be a group, or, more generally, a Bol loop. The binary algebra $(G,*)$ with \[a*b=a\cdot b^{-1}a\] is an involutory quandle, called the \emph{core} of $(G,\cdot)$. The core is a quasigroup if and only if the loop is uniquely 2-divisible \cite[Theorem 9.4]{Bel}. 
The core operation was introduced by Bruck who proved that isotopic Moufang loops have isomorphic cores \cite{Bru}. It was later picked up by Belousov and others to construct some of the first examples of involutory left distributive quasigroups, see e.g. \cite[Chapter IX]{Bel} or \cite{Uma}.

\begin{example}\label{e:ildq15}
The smallest non-medial involutory left distributive quasigroup has order 15 and it is the core of the B-loop constructed in Example \ref{e:bloop15}. Explicitly, it is the quasigroup $(\Z_5\times\Z_3,*)$ with \[(a,x)*(b,y)=(\mu_{x,y}a-b,-x-y)\]
where $\mu_{x,y}\in\Z_5^*$ are given by the following table:
\begin{displaymath}
    \begin{array}{c|ccc}
          & 0 & 1 & 2 \\\hline
        0 & 2 & -1 & -1 \\
        1 & -1 & 2 & -1 \\
        2 & -1 & -1 & 2 \\
    \end{array}
\end{displaymath}
\end{example}

\section{Left distributive quasigroups: Isotopy}\label{sec:ldq-isotopy}

\subsection{Right linear representation}

Restricting self-distributivity to only one side, it is natural to expect that the loop counterpart will admit one of the weaker one-sided loop conditions mentioned in Section \ref{ssec:loops}. There are good news and bad news. Left distributive quasigroups are polynomially equivalent to a certain class of ``non-associative modules", satisfying a (very) weak associative law. However, the connection is non-linear (only one of the defining mappings is an automorphism), and the corresponding class of loops, called \emph{Belousov-Onoi loops} here, extends beyond the well-established theories (except for some special cases). The correspondence is therefore of limited utility at the moment. Nevertheless, it is interesting to look at details. 
Most of the ideas of the present section were discovered by Belousov and Onoi \cite{BO}, but our presentation is substantially different.

Let $(Q,\cdot)$ be a loop and $\psi$ its automorphism. We will call $(Q,\cdot,\psi)$ a \emph{Belousov-Onoi module} (shortly, \emph{BO-module}) if
\[\varphi(ab)\cdot\psi(ac)=a\cdot\varphi(b)\psi(c)\tag{BO}\]
holds for every $a,b,c\in Q$, where $\varphi(x)=x\rdd\psi(x)$ is the \emph{companion mapping} for~$\psi$.
(The explanation why is it reasonable to consider such structures as ``non-associative modules" has been explained at the end of Section \ref{ss:linear}.)
To match the identity (BO) to the Bol identity, substitute $\psi^{-1}(ac)$ for $c$ and obtain an equivalent identity
\[\varphi(ab)\cdot(\psi(a)\cdot ac)=a\cdot(\varphi(b)\cdot ac).\tag{BO'}\]

\begin{example}\label{e:b-module}
We state a few examples of Belousov-Onoi modules.
\begin{enumerate}
	\item Every loop $(Q,\cdot)$ turns into the BO-module $(Q,\cdot,id)$. If $\psi(x)=x$, then $\varphi(x)=1$ and thus the identity (BO) holds.
	\item Every group $(Q,\cdot)$ with any automorphism $\psi$ turns into the BO-module $(Q,\cdot,\psi)$. Condition (BO) is easily verified.
	\item Every Bruck loop $(Q,\cdot)$ turns into the BO-module $(Q,\cdot,^{-1})$. If $\psi(x)=x^{-1}$, then $\varphi(x)=x^2$ by power-associativity, and we verify (BO') by $(ab)^2\cdot(a^{-1}\cdot ac)=(ab)^2\cdot c=a\cdot(b^2\cdot ac)$ using Lemma \ref{l:bruck} in the second step.
\end{enumerate}
\end{example}

Call a BO-module \emph{non-trivial} if $\psi\neq id$. There are relatively few loops that turn into a non-trivial BO-module, see the values of $BOM(n)$ in Table \ref{t:b-module}. Nevertheless, nearly all groups and all Bruck loops (except possibly those where $x^{-1}=x$) have the property.

A BO-module turns naturally into a quandle. The proof illustrates very well the conditions imposed by the definition.

\begin{proposition}\label{p:ldq1}
Let $(Q,\cdot,\psi)$ be a Belousov-Onoi module, $\varphi$ the companion mapping, and define for every $a,b\in Q$ \[a*b=\varphi(a)\cdot\psi(b).\] 
Then $(Q,*)$ is a quandle. The quandle is a quasigroup if and only if $\varphi$ is a permutation.
\end{proposition}

\begin{proof}
Idempotence explains the definition of the companion mapping: we have $a*a=a$ iff $\varphi(a)\cdot\psi(a)=a$ iff $\varphi(a)=a\rdd\psi(a)$.

Unique left division follows from the fact that $\psi$ is a permutation: we have $a*x=\varphi(a)\cdot\psi(x)=b$ iff $\psi(x)=\varphi(a)\ldd b$ iff $x=\psi^{-1}(\varphi(a)\ldd b)$.

Left distributivity is verified as follows: expanding the definition of $*$ and using the identity (BO), we obtain
\[(a*b)*(a*c)=\varphi(\varphi(a)\psi(b))\cdot\psi(\varphi(a)\psi(c))=\varphi(a)\cdot(\varphi\psi(b)\cdot\psi^2(c)),\] 
and since $\psi$ is an automorphism and $\varphi$ a term operation, we have $\varphi\psi=\psi\varphi$, and thus the right hand side equals \[\varphi(a)\cdot(\psi\varphi(b)\cdot\psi^2(c))=\varphi(a)\cdot\psi(\varphi(b)\psi(c))=a*(b*c).\]

Unique right division is dual to the left case: it happens if and only if $\varphi$ is a permutation.
\end{proof}

\begin{example}\label{e:b-module->q}
Consider the three items from Example \ref{e:b-module}.
\begin{enumerate}
	\item Any trivial BO-module $(Q,\cdot,id)$ results in a projection quandle $(Q,*)$ with $a*b=b$.
	\item The BO-module $(Q,\cdot,\psi)$, constructed over a group with an automorphism, results in a homogeneous quandle $(Q,*)$ with \[a*b=a\psi(a^{-1}b).\] 
	If $Q$ is finite, then $(Q,*)$ is a quasigroup if and only if $\psi$ is a regular automorphism (i.e. the unit is the only fixed point of $\psi$).
	Belousov \cite[Theorem 9.2]{Bel} proves that all left distributive quasigroups isotopic to a group result in this particular way, and Galkin \cite[Section 5]{Gal-ldq} shows a number of interesting properties of such quasigroups. See Construction \ref{c:homog} for a generalization of this idea which covers all left distributive quasigroups.
	\item The BO-module $(Q,\cdot,^{-1})$, constructed over a Bruck loop, results in an involutory quandle $(Q,*)$ with $a*b=a^2b^{-1}$. It follows from Lemma \ref{l:bruck}(2) that $x\mapsto x^2$ is a homomorphism from $(Q,*)$ to the core of $(Q,\cdot)$; hence, if $(Q,\cdot)$ is a B-loop, then the two constructions result in isomorphic quasigroups. In Theorem \ref{thm:lsldq}, we shall see that all involutory left distributive quasigroups result this way.
\end{enumerate}
\end{example}

Relatively few quandles admit a \emph{Belousov-Onoi representation} as in Proposition \ref{p:ldq1}, see the values of $BOQ(n)$ in Table \ref{t:b-module}. Even connected quandles do not always result from a BO-module: for example, a quick computer search reveals that none of the quandles constructed over a BO-module of order 6 is connected (compare to \cite[Table 2]{HSV}).
In the latin case, however, the situation is different. The setting of BO-modules was designed by Belousov and Onoi in order to prove that all left distributive quasigroups (latin quandles) admit a representation as in Proposition \ref{p:ldq1}.

A loop $(Q,\cdot)$ possesing an automorphism $\psi$ such that $(B,\cdot,\psi)$ is a BO-module and the companion mapping for $\psi$ is a permutation, will be called a \emph{Belousov-Onoi loop} (shortly, \emph{BO-loop}) with respect to $\psi$. (The original name was \emph{S-loops}, for no apparent reason. Our definition uses the characterizing condition of \cite[Theorem 4]{BO}.)

\begin{proposition}[{\cite{BO}}]\label{p:ldq2}
Let $(Q,*)$ be a left distributive quasigroup, $e\in Q$ and let \[a\cdot b=(a/e)*(e\ld b).\]
Then $(Q,\cdot)$ is a Belousov-Onoi loop with respect to $\psi=L_e$, the companion mapping is $\varphi=R_e$ and \[a*b=\varphi(a)\cdot\psi(b).\] 
Moreover, different choices of $e$ result in isomorphic loops.
\end{proposition}

\begin{proof}
First notice that $a*b=(a*e)\cdot(e*b)=\varphi(a)\cdot\psi(b)$. Indeed, both $\varphi,\psi$ are permutations and $\varphi$ is the companion for $\psi$, since $\varphi(a)\cdot\psi(a)=a$. To prove that $\psi$ is an automorphism of $(Q,\cdot)$, we calculate for every $a,b\in Q$
\begin{align*}
\psi(ab)=e*ab&=e*((a/e)*(e\ld b))\\&=(e*(a/e))*(e*(e\ld b))\\&=((e*a)/e)*(e\ld(e*b)=(e*a)\cdot(e*b)=\psi(a)\psi(b).
\end{align*}
In the third and fourth steps, we used left distributivity: in the latter case, since $L_e$ is an automorphism of $(Q,*)$, we also have $L_e(x/y)=L_e(x)/L_e(y)$ for every $x,y$. To prove the condition (BO), we calculate for every $a,b\in Q$
\begin{align*}
\varphi(ab)\cdot\psi(ac)&=(ab*e)\cdot(e*ac)=ab*ac\\&=((a/e)*(e\ld b))*((a/e)*(e\ld c))\\&=(a/e)*((e\ld b)*(e\ld c))\\&=(a/e)*(e\ld(b*c))=a\cdot(b*c)=a\cdot\varphi(b)\psi(c).
\end{align*}
In the fourth and fifth steps, we used left distributivity: in the latter case, using the fact that $L_e^{-1}$ is also an automorphism of $(Q,*)$.

Let $e_1,e_2\in Q$ and consider an automorphism $\rho$ of $(Q,*)$ such that $\rho(e_1)=e_2$ (for example, we can take $\rho=L_{e_2/e_1}$). Then 
$\rho$ is an isomorphism of the corresponding loops $(Q,\cdot_1)$ and $(Q,\cdot_2)$, since 
\[\rho(a\cdot_1 b)=\rho((a/e_1)*(e_1\ld b))=(\rho(a)/\rho(e_1))*(\rho(e_1)\ld \rho(b))=\rho(a)\cdot_2\rho(b)\]
for every $a,b\in Q$.
\end{proof}

If $(Q,\cdot)$ is a Belousov-Onoi loop with respect to $\psi$, the companion mapping $\varphi$ is usually not an automorphism. 
In such a case, the representation of $(Q,*)$ over $(Q,\cdot)$ will be called \emph{right linear}. 
In Proposition \ref{p:gbl-cml}, we shall prove that $\varphi$ is an automorphism if and only if the loop is commutative Moufang.
Therefore, according to Theorem \ref{thm:dq}, we do not have a linear representation, unless we handle a (both-side) distributive quasigroup. 

Still, the left distributive quasigroup $(Q,*)$ (formally, the algebra $(Q,*,\ld,/)$) is \emph{polynomially equivalent} to the Belousov-Onoi module $(Q,\cdot,\psi)$ (formally, the algebra $(Q,\cdot,\ldd,\rdd,\psi,\psi^{-1})$): all operations in Proposition \ref{p:ldq2} were defined polynomially, the same can be shown about the division operations, and $\varphi(x)=x\rdd\psi(x)$ is a polynomial, too. In fact, we can think of the mapping $\varphi$ as \emph{quadratic} over the BO-module $(Q,\cdot,\psi)$, as the variable $x$ appears only twice in its definition.

Combining Propositions \ref{p:ldq1} and \ref{p:ldq2}, we can formulate the following representation theorem. 

\begin{theorem}[\cite{BO}]\label{thm:ldq}
The following are equivalent for a quasigroup $(Q,*)$:
\begin{enumerate}
	\item it is left distributive;
	\item it is right linear over a Belousov-Onoi loop (with respect to the automorphism used in the right linear representation).
\end{enumerate}
\end{theorem}

\begin{example}\label{e:15}
The smallest non-associative Belousov-Onoi loops have order 15, and there are two of them (up to isomorphism). One is a B-loop, see Example \ref{e:bloop15}. The other one can be constructed by a modification of the previous construction.
Consider the loop $(\Z_5\times\Z_3,\cdot)$ with \[(a,x)\cdot(b,y)=(\varphi_{x,y}a+b+\theta_{x,y},x+y)\]
where $\varphi_{x,y}\in\Z_5^*$ are as before, and $\theta_{x,y}\in\Z_5$ are given by the following table:
\begin{displaymath}
    \begin{array}{c|ccc}
          & 0 & 1 & 2 \\\hline
        0 & 0 & 0 & 0 \\
        1 & 0 & -1 & 1 \\
        2 & 0 & -2 & 2 \\
    \end{array}
\end{displaymath}
It is straightforward to check that this is a BO-loop with respect to the automorphism $(a,x)\mapsto(-a+\delta_{x,2},-x)$ where $\delta_{x,y}=1$ if $x=y$ and $\delta_{x,y}=0$ otherwise. It is not a B-loop, it does not even have the LIP. It is also an abelian extension of $\Z_5$ by $\Z_3$. If we set $\theta_{x,y}=0$ for every $x,y$, we would have obtained the B-loop of Example \ref{e:bloop15}.

Correspondingly, the smallest non-medial left distributive quasigroups have order 15, and there are two of them (up to isomorphism). One is involutory, see Example \ref{e:ildq15}. The other one can be constructed as $(\Z_5\times\Z_3,*)$ with \[(a,x)*(b,y)=(\mu_{x,y}a-b+\tau_{x,y},-x-y)\]
where $\mu_{x,y}\in\Z_5^*$ is as before, and $\tau_{x,y}=\delta_{x-y,1}$ for every $x,y$. (See \cite{CEHSY,CH} for a generalization of this construction, originally suggested by Galkin \cite{Gal-small}.)
\end{example}

\subsection{Belousov-Onoi loops}\label{ss:gbl}

Given the correspondence of Theorem \ref{thm:ldq}, a natural question arises. What are these Belousov-Onoi loops? Can we use an established part of loop theory to investigate left distributive quasigroups? The current state of knowledge is unsatisfactory in this respect. In the rest of the section, we summarize most of the known results on BO-loops.

First of all, it is not even clear how to construct Belousov-Onoi loops which are not B-loops. All BO-loops of order less than $15$ are abelian groups, and there are two non-associative BO-loops of order 15, see Example \ref{e:15}. Nowadays, these facts are easy to check on a computer, but back in the 1970s, this was realized only indirectly, via Theorem \ref{thm:ldq}, using the theory of left distributive quasigroups. The first example of a left distributive quasigroup not isotopic to any Bol loop was constructed by Onoi in \cite{Onoi-homog}. The construction is quite intricate, and occupies a major part of the paper: Onoi starts with $2\times 2$ matrices over a certain non-associative ring with four elements, takes a quadratic operation on pairs of the matrices, and then creates a left distributive isotope; thus, the quasigroup has order $2^{16}$. The smallest example, of order 15, was found later by Galkin in \cite{Gal-small}. We see the situation twisted: it is not the loops that reveal properties of the quasigroups, it is the other way around!

Table \ref{t:b-module} shows some enumeration results related to Belousov-Onoi loops. The upper part compares the numbers $L(n)$ of all loops, $BOM(n)$ of loops that turn into a non-trivial BO-module, and $BOL(n)$ of BO-loops, of order $n$ up to isomorphism. The lower part compares the numbers $Q(n)$ of all quandles, $BOQ(n)$ of quandles that admit a Belousov-Onoi representation as in Proposition \ref{p:ldq1}, and $LQ(n)$ of latin quandles (left distributive quasigroups), of order $n$ up to isomorphism. The sequences $L(n)$, $Q(n)$ are well known \cite{OEIS}, the other numbers were calculated using an exhaustive computer search.

\begin{table}[ht]
\[
\begin{array}{r|rrrrrrrr}
n &      1&2&3&4&5&6&7&8\\\hline
L(n)&    1&1&1&2&6&109&23746&106228849  \\
BOM(n)&   0&0&1&1&1&3&1&144 \\
BOL(n)&   1&0&1&1&1&0&1&3 \\\hline
Q(n)&    1&1&3&7&22&73&298&1581 \\
BOQ(n)&   1&1&2&3&4&3&6&9\\
LQ(n)&   1&0&1&1&3&0&5&2 
\end{array}
\]
\caption{Enumeration of small loops and quandles related to the Belousov-Onoi representation.}
\label{t:b-module}
\end{table}

\medskip
In the rest of the section, we present a few results that relate the Belousov-Onoi loops to more established classes of loops, and specialize the correspondence between left distributive quasigroups and Belousov-Onoi loops, proved in Theorem \ref{thm:ldq}, on two important subclasses: the distributive quasigroups, and the involutory left distributive quasigroups.

We start with a variation on {\cite[Theorem 2]{Onoi-SM}}. Our proof, based on Theorem \ref{t:cml} (the Pflugfelder's part), is much simpler.

\begin{proposition}\label{p:gbl-cml}
Let $(Q,\cdot)$ be a loop, $\psi$ an automorphism of $(Q,\cdot)$ and assume its companion mapping $\varphi$ is a permutation. Then any two of the following properties imply the third:
\begin{itemize}
	\item $(Q,\cdot)$ is a Belousov-Onoi loop with respect to $\psi$;
	\item $(Q,\cdot)$ is a commutative Moufang loop;
	\item $\varphi$ is an automorphism.
\end{itemize}
\end{proposition}

\begin{proof}
According to Theorem \ref{t:cml}, $(Q,\cdot)$ is a commutative Moufang loop if and only if, for some mapping $f$ on $Q$, the identity 
$f(x)y\cdot xz=f(x)x\cdot yz$ holds. Let $f=\varphi\psi^{-1}$ and substitute $x=\psi(a)$, $y=\varphi(b)$, $z=\psi(c)$. We obtain that $(Q,\cdot)$ is a commutative Moufang loop if and only if $\varphi(a)\varphi(b)\cdot\psi(a)\psi(c)=\varphi(a)\psi(a)\cdot\varphi(b)\psi(c)=a\cdot\varphi(b)\psi(c)$ for every $a,b,c\in Q$. Consider the following three expressions:
\begin{align*}
X&=\varphi(a)\varphi(b)\cdot\psi(a)\psi(c)\\
Y&=a\cdot\varphi(b)\psi(c)\\
Z&=\varphi(ab)\cdot\psi(a)\psi(c)
\end{align*}
We just proved that $X=Y$ for every $a,b,c\in Q$ iff $(Q,\cdot)$ is commutative Moufang.
According to condition (BO), $Y=Z$ for every $a,b,c\in Q$ iff $(Q,\cdot)$ is a BO-loop with respect to $\psi$.
And, obviously, $X=Z$ for every $a,b,c\in Q$ iff $\varphi$ is an automorphism of $(Q,\cdot)$.
\end{proof}

Now we can reprove Belousov's result that every distributive quasigroup is linear over a commutative Moufang loop (a similar argument is presented in \cite[Theorem 3]{Onoi-SM}).

\begin{proof}[Proof of Theorem \ref{thm:dq}, $(3)\Rightarrow(4)$] 
Let $(Q,*)$ be a distributive quasigroup, pick $e\in Q$ a let $a\cdot b=(a/e)*(e\ld b)$. Since $(Q,*)$ is left distributive, $(Q,\cdot)$ is a BO-loop with respect to $L_e$, which in turn is an automorphism of $(Q,\cdot)$. Since $(Q,*)$ is right distributive, $(Q,\cdot)$ is also a right(!) BO-loop (this is irrelevant for us) with respect to $R_e$, which in turn is an automorphism of $(Q,\cdot)$. We showed that the companion of $L_e$ is an automorphism, hence $(Q,\cdot)$ is a commutative Moufang loop by Proposition \ref{p:gbl-cml}.
\end{proof}

Next we show that B-loops are precisely the BO-loops with respect to the left inverse mapping. 

\begin{proposition}[{\cite[Theorem 8]{BO}}]\label{p:gbl-bl}
Let $(Q,\cdot)$ be a loop and $\psi(x)=x\ldd 1$. Then $(Q,\cdot)$ is a Belousov-Onoi loop with respect to $\psi$ if and only if it is a B-loop.
\end{proposition}

\begin{proof}
The backward implication was proved in Example \ref{e:b-module}(3). In the forward direction, condition (BO) with $b=1$ and $c=a$ says that $\varphi(a)\psi(a^2)=a\psi(a)=1$, and thus \[\varphi(a)=1\rdd\psi(a^2)=1\rdd(a^2\ldd1)=a^2\] for every $a\in Q$. Hence, $(Q,\cdot)$ is a uniquely 2-divisible loop with the LAIP. Now, condition (BO), upon substitution of $\psi^{-1}(c)$ for $c$, says that $(ab)^2\cdot((a\ldd1)\cdot c)=a\cdot b^2c$, and we can use Theorem \ref{t:bruck} to conclude that $(Q,\cdot)$ is a Bol loop.
\end{proof}

%A way around proof of the previous propositon, using the corresponding left distributive quasigroups, was presented in \cite[Theorem 8]{BO}.

With the aid of Proposition \ref{p:gbl-bl}, we establish the correspondence between involutory left distributive quasigroups and B-loops. 
This connection has a rich history: it was first realized by Robinson in his 1964 PhD thesis, but published only 15 years later in \cite{Rob}. Independently, Belousov and Florya \cite[Theorem 3]{BF} noticed that involutory left distributive quasigroups are isotopic to Bol loops, but they did not formulate the full correspondence.
Independently, the theorem was formulated by Kikkawa \cite{Kik1} (at the first glance, it is not obvious that his loop axioms are equivalent to those of B-loops, as he uses condition (2') of Theorem \ref{t:bruck} instead of the Bol identity). The theorem was rediscovered once more in \cite[Theorems 2.5 and 2.7]{NS}. 
Unlike all of the other representation theorems in the present paper, Theorem \ref{thm:lsldq} has a fairly straightforward direct proof, and contemporary ATP systems can prove it within a second. 

\begin{theorem}[\cite{Kik1,NS,Rob}]\label{thm:lsldq}
The following are equivalent for a quasigroup $(Q,*)$:
\begin{enumerate}
	\item it is involutory left distributive;
	\item there is a B-loop $(Q,\cdot)$ such that $a*b=a^2\cdot b^{-1}$.
\end{enumerate}
\end{theorem}

\begin{proof}
$(1)\Rightarrow(2)$ Consider the quasigroup operation $a\cdot b=(a/e)*(e\ld b)$. According to Theorem~\ref{thm:ldq}, $(Q,\cdot)$ is a BO-loop with respect to $L_e$. If we prove that $L_e(x)=x\ldd 1$, Proposition~\ref{p:gbl-bl} applies and $(Q,\cdot)$ is a B-loop. Then, clearly, the companion mapping is $\varphi(x)=x^2$, and thus $a*b=a^2\cdot b^{-1}$.

We need to check that $L_e(a)=e*a$ equals $a\ldd 1=a\ldd e$ for every $a\in Q$. We have $e*a=a\ldd e$ iff $a\cdot(e*a)=e$ iff $(a/e)*a=e$ (we expanded the definition of $\cdot$). Now multiply the last identity by $a/e$ from the left, and obtain $(a/e)*((a/e)*a)=(a/e)*e=a$, which is always true thanks to the involutory law.

$(2)\Rightarrow(1)$ Left distributivity was verified in Proposition \ref{p:ldq1} through Example \ref{e:b-module}(3). It is involutory, as $a*(a*b)=a^2(a^2b^{-1})^{-1}=a^2(a^{-2}b)=b$ thanks to the AIP and LIP in Bruck loops.
\end{proof}

As far as we know, only two papers, \cite{BO,Onoi-SM}, are devoted to Belousov-Onoi loops.
We state two more results here. The first one identifies some important subclasses of BO-loops, see \cite[Theorem 2]{BO}, \cite[Theorem 1]{Onoi-SM} and \cite[Theorem 3]{BO}, respectively.

\begin{proposition}[\cite{BO,Onoi-SM}]\label{p:gbl3}
Let $(Q,\cdot)$ be a Belousov-Onoi loop.
\begin{enumerate}
	\item It is Bol if and only if it is left alternative.
	\item It is Moufang iff it is right alternative, iff it has the RIP, iff the identity $(xy)^{-1}=y^{-1}x^{-1}$ holds, iff the identity $x\cdot yx=xy\cdot x$ holds.
	\item It is a group if and only if it is left alternative and every square is nuclear. % left nuclear.
\end{enumerate}
\end{proposition}

The second is a characterization of Belousov-Onoi loops that matches well with Theorem \ref{t:bruck} on B-loops.

\begin{theorem}[\cite{BO}]\label{t:gbruck}
The following are equivalent for a loop $(Q,\cdot)$ with an automorphism $\psi$ such that its companion mapping $\varphi$ is a permutation:
\begin{enumerate}
	\item[(1)] it satisfies the identity $\varphi(x)\cdot\psi(x)y=xy$ and it is left automorphic as a BO-module (i.e. the left inner mappings are automorphisms of $(Q,\cdot,\psi)$);
	\item[(1')] the identities $\varphi(x)\cdot\psi(x)y=xy$ and $L_{x,y}\psi=\psi L_{x,y}$ hold;
	\item[(2)] it satisfies condition (BO).
\end{enumerate}
\end{theorem}

\begin{proof}[Proof sketch]
The equivalence of (1') and (2) is proved in \cite[Theorem 4]{BO}. Condition (1') is a special case of (1). It remains to prove that in any BO-loop $(Q,\cdot)$, every inner mapping $L_{x,y}$ is an automorphism of $(Q,\cdot,\psi)$. It respects $\psi$ as postulated in (1'). According to Theorem \ref{thm:ldq}, $(Q,\cdot)$ is isotopic to a left distributive quasigroup, and Belousov and Florya prove in \cite[Theorem~2]{BF} that every loop isotope of a left distributive quasigroup (actually, more generally, of any F-quasigroup) is left automorphic.
\end{proof}

We are not aware of any general structural results on left distributive quasigroups proved using the correspondence of Theorem \ref{thm:ldq}. Actually, with the efficient methods we will describe in Section~\ref{sec:ldq-homog}, the correspondence could be used in the other direction, to investigate properties of Belousov-Onoi loops via left distributive quasigroups.

Nevertheless, in the involutory case, loop theory helps considerably, as the theory of Bruck loops is well developed. One example for all: Glauberman proved that finite B-loops are solvable, and that analogies of the Lagrange and Sylow theorems hold (see \cite[Section 8]{Gla} for precise statements). 
Since a B-loop $(Q,\cdot)$ and its corresponding involutory left distributive quasigroup $(Q,*)$ are polynomially equivalent, they share all the properties defined by polynomial operations. For instance, congruences and solvability. The polynomial correspondence uses a single constant, $e$, therefore, the subloops of $(Q,\cdot)$ are exactly the subquasigroups of $(Q,*)$ containing $e$. Since $e$ can be chosen arbitrarily, the Lagrange and Sylow properties are shared by $(Q,*)$ as well. In Section \ref{ss:ldq-structure}, we put these results into a broader context.

\section{Left distributive quasigroups: Homogeneous representation}\label{sec:ldq-homog}

\subsection{Homogeneous representation}\label{ss:ldq-homog}

Our exposition in this section follows our recent paper \cite{HSV} where many older ideas are collected and adjusted to the modern quandle setting. A reader interested in more details (proofs in particular), is recommended to consult \cite{HSV}. Here we try to reference the original sources.

Recall that a quandle $Q$ is \emph{homogeneous}, if $\aut(Q)$ acts transitively on $Q$. Since $\lmlt(Q)$ is a subgroup of $\aut(Q)$, all connected quandles (and thus all left distributive quasigroups) are homogeneous. 

It is not clear who came up with Construction \ref{c:homog}. But it was certainly Galkin \cite{Gal-ldq} who recognized its importance for representing self-distributive algebraic structures, followed independently by Joyce and others (perhaps a partial credit could be paid to Loos \cite{Loos}, too). 

\begin{construction}[\cite{Gal-ldq,Joy}]\label{c:homog}
Let $(G,\cdot)$ be a group, $H$ its subgroup, and $\psi$ an automorphism of $(G,\cdot)$ such that $\psi(a)=a$ for every $a\in H$. Such a triple $(G,H,\psi)$ will be called \emph{admissible}. Denote $G/H$ the set of left cosets $\{aH:a\in G\}$, and consider the binary algebra $\Q(G,H,\psi)=(G/H,*)$ with \[aH*bH=a\psi(a^{-1}b)H.\]
It is straightforward to verify that $\Q(G,H,\psi)$ is a homogeneous quandle.
If $G$ is finite, then $\Q(G,H,\psi)$ is a quasigroup if and only if, for every $a,u\in G$, $a\psi(a^{-1})\in H^u$ implies $a\in H$.
% \cite[Theorem 4.2]{Gal-ldq}.
\end{construction}

Note that the operation can be written as $aH*bH=\varphi(a)\psi(b)H$, where $\varphi$ is the companion mapping to $\psi$, so this really is, in a way, a variation on the isotopy method.
Also note that the special case $\Q(G,1,\psi)$, with the trivial subgroup $H=1$, is the same construction as in Example \ref{e:b-module->q}(2).

\begin{example}
According to Theorem \ref{thm:medial}, medial idempotent quasigroups are precisely the quasigroups $\Q(G,1,\psi)$ where $G$ is an abelian group and $\psi$ is an automorphism such that its companion is a permutation (and therefore an automorphism, too).
\end{example}

In the present section, we will denote conjugation as $a^b=bab^{-1}$ (unlike most texts on group theory, we use the right-left composition of mappings, hence it is natural to use the dual notation for conjugation). Similarly, we will denote $a^G=\{a^g:g\in G\}$ the conjugacy class of $a$ in $G$, $H^b=\{h^b:h\in H\}$, and $-^b$ the mapping $x\mapsto x^b$.
If $G$ is a group acting on a set $X$ and $e\in X$, we will denote $e^G$ the orbit containing $e$, and $G_e$ the stabilizer of $e$. 

The following observation appeared in many sources in various forms, its complete proof can be found e.g. in \cite[Section 3]{HSV}.

\begin{proposition} %[{\cite[Section 3]{HSV}}]
\label{p:homog}
Let $(Q,*)$ be a quandle and $e\in Q$. Let $G$ be a normal subgroup of $\aut(Q,*)$.
Then $(G,G_e,-^{L_e})$ is an admissible triple and the orbit subquandle $(e^G,*)$ is isomorphic to the quandle $\Q(G,G_e,-^{L_e})$.
\end{proposition}

\begin{proof}[Proof sketch]
Since $-^{L_e}$ is a restriction of an inner automorphism to a normal subgroup, it is an automorphism of $G$. It is straightforward to check that it fixes the stabilizer pointwise. Consider the bijective mapping $f:G/G_e\to e^G$, $\alpha G_e\mapsto \alpha(e)$. Again, it is straightforward to check that this is a quandle isomorphism $\Q(G,G_e,-^{L_e})\simeq(e^G,*)$.
\end{proof}

Consider three particular choices of the normal subgroup: $G=\aut(Q,*)$, $G=\lmlt(Q,*)$ and $G=\lmlt(Q,*)'$, respectively. If $G$ acts transitively on $Q$, Proposition \ref{p:homog} claims the following:
%(The names date back to Galkin \cite{Gal-ldq}.)
\begin{itemize}
	\item Every homogeneous quandle $(Q,*)$ is isomorphic to $\Q(G,G_e,-^{L_e})$ with $G=\aut(Q,*)$.
	\item Every connected quandle $(Q,*)$ is isomorphic to $\Q(G,G_e,-^{L_e})$ with $G=\lmlt(Q,*)$. This will be called the \emph{canonical representation} of $(Q,*)$. 
	\item Every connected quandle $(Q,*)$ is isomorphic to $\Q(G,G_e,-^{L_e})$ with $G=\lmlt(Q,*)'$. This will be called the \emph{minimal representation} of $(Q,*)$. (To make it work, one has to show that the actions of $\lmlt(Q,*)$ and $\lmlt(Q,*)'$ have identical orbits \cite{Gal-ldq,Joy}.)
\end{itemize}

\begin{corollary}[{\cite[Theorem 7.1]{Joy}}]
A quandle is isomorphic to $\Q(G,H,\psi)$ for some admissible triple $(G,H,\psi)$ if and only if it is homogeneous.
\end{corollary}

Why minimal representation? Galkin \cite[Theorem 4.4]{Gal-ldq} proved the following fact: if a connected quandle $(Q,*)$ is isomorphic to $\Q(G,H,\psi)$ for some admissible triple $(G,H,\psi)$, then $\lmlt(Q)'$ embeds into a quotient of $G$. Hence, if $Q$ is finite, the minimal representation is the one with the smallest group $G$.

Why canonical representation? Fix a set $Q$ and an element $e$. We have a 1-1 correspondence between connected quandles $(Q,*)$ on one side, and certain configurations in transitive groups acting on $Q$ on the other side. A \emph{quandle envelope} is a pair $(G,\zeta)$ where $G$ is a transitive group on $Q$ and $\zeta\in Z(G_e)$ (here $Z$ denotes the center) such that $\langle\zeta^G\rangle=G$. The correspondence is given by the following two mutually inverse mappings:
\begin{align*}
\text{connected quandle}&\ \leftrightarrow\ \text{quandle envelope}\\
(Q,*)&\ \rightarrow\  (\lmlt(Q,*),L_e)\\
\Q(G,G_e,-^{\zeta})&\ \leftarrow\  (G,\zeta)\\
\end{align*}
If $Q$ is finite, then an envelope $(G,\zeta)$ corresponds to a latin quandle if and only if $\zeta^{-1}\zeta^\alpha$ has no fixed point for every $\alpha\in G\smallsetminus G_e$.
Moreover, two envelopes $(G_1,\zeta_1)$ and $(G_2,\zeta_2)$ yield isomorphic quandles if and only if there is a permutation $f$ of $Q$ such that $f(e)=e$, $\zeta_1^f=\zeta_2$ and $G_1^f=G_2$ (in particular, the two groups are isomorphic).
See \cite[Section 5]{HSV} for details, and \cite[Section 7]{HSV} for a plenty of illustrative examples (the correspondence seems to be an original contribution of the paper).

Canonical representation is arguably the most powerful tool currently available to study connected quandles, and left distributive quasigroups in particular, as we shall see in the remaining part of the section.

\subsection{Enumeration}\label{ss:ldq-enumeration}

Canonical representation allows to enumerate connected quandles (left distributive quasigroups in particular) with $n$ elements, provided a classification of transitive groups of degree $n$. Currently, such a library is available for $n\leq47$. The enumeration of small connected quandles was carried out in \cite{HSV,Ven}. Here, in Table \ref{t:ldq}, we present the numbers of quasigroups, where $LD(n)$ refers to non-medial left distributive ones, and $ILD(n)$ to non-medial involutory left distributive ones, of order $n$ up to isomorphism. 
We recall from Section \ref{ss:dq-structure} that $MI(n)$ denotes the number of medial idempotent quasigroups and can be determined by Hou's formulas \cite{Hou}.

\begin{table}[ht]
\[
\begin{array}{r|rrrrrrrrrrrrrrrr}
n &      1&2&3&4&5&6&7&8&9&10&11&12&13&14&{\bf 15}&16\\\hline
LD(n)&    0&0&0&0&0&0&0&0&0&0&0&0&0&0&{\bf 2}&0 \\
ILD(n)&   0&0&0&0&0&0&0&0&0&0&0&0&0&0&{\bf 1}&0 \\\hline
MI(n)&    1&0&1&1&3&0&5&2&8&0 &9 & 1&11&0 &{ 3 }&9  \\\\
n       &17&18&19&20&{\bf 21}&22&23&24&25&26&{\bf 27}&{\bf 28}&29&30&31&32\\\hline
LD(n) & 0&0& 0&0 &{\bf 2} &0 & 0&0 & 0&0 &{\bf 32}&{\bf 2} & 0&0 & 0&0  \\
ILD(n)& 0&0& 0&0 &{\bf 1} &0 & 0&0 & 0&0 &{\bf  4}&{\bf 0} & 0&0 & 0&0  \\\hline
MI(n) &15&0&17&3 &{ 5} &0 &21&2 &34&0 &{ 30}&{ 5} &27&0 &29&8  \\\\
       &{\bf 33}&34&35&{\bf 36}&37&38&{\bf 39}&40&41&42&43&44&{\bf 45}&46&47\\\hline
LD(n) &{\bf  2}&0 & 0&{\bf 1}& 0&0 &{\bf  2}&0 & 0&0 & 0&0 &{\bf 12}&0 & 0\\
ILD(n) &{\bf  1}&0 & 0&{\bf 0}& 0&0 &{\bf  1}&0 & 0&0 & 0&0 &{\bf 3}&0 & 0\\\hline
MI(n)  &{ 9} &0 &15&{ 8}   &35&0 &{ 11}&6 &39&0 &41&9 &{ 24}&0 &45\\
\end{array}
\]
\caption{Enumeration of small left distributive quasigroups.}
\label{t:ldq}
\end{table}

From the historical perspective, the first serious attempt on enumeration was carried out by Galkin \cite{Gal-small} who calculated (without a computer!) the numbers $LD(n)$ for $n<27$, and found that $LD(27)\geq3$. A few results in the involutory case can be found in an earlier paper by Nobusawa \cite{Nob1}. In \cite{Ven}, Vendramin enumerated connected quandles of size $n\leq35$, which was the state-of-the-art in the classification of transitive groups at the time, but his algorithm works for larger orders as well.
%In the involutory case, an earlier paper by Nobusawa \cite{Nob1} demonstrates that $ILD(p)=ILD(p^2)=ILD(2k)=0$ for every prime $p$ and natural number $k$, and gave a construction showing that $ILD(27)\geq1$.

One can make a few observations about Table \ref{t:ldq}. 
Most obviously, we do not see any left distributive quasigroups (medial or not) with $4k+2$ elements. This is true for every $k$, as proved by Stein already in the 1950s \cite[Theorem 9.9]{Ste-found}.

\begin{theorem}[\cite{Ste-found}]
There are no left distributive quasigroups of order $4k+2$, for any $k\geq0$.
\end{theorem}

The fact is easy to observe in the medial case: any medial idempotent quasigroup of order $4k+2$ is linear over an abelian group which is the direct product of $\Z_2$ and a group of odd order; however, there is no idempotent quasigroup of order 2. Stein's remarkable argument uses a topological reasoning, constructing a triangulated polyhedron from the graph of the quasigroup and discussing parity of its Euler characteristic (for details, see \cite{Ste-found} or  \cite[Section 6]{Gal-survey}). In \cite{Ste-homo}, Stein observed that the result extends to all homogeneous quasigroups, since each of them is isotopic to an idempotent quasigroup and the same method as in the self-distributive case proves non-existence. In \cite[Theorem 6.1]{Gal-ldq}, Galkin proved Stein's theorem using a short group theoretical argument about the minimal representation. 

Let us note that connected quandles of order $4k+2$ do exist, although there are no connected quandles with $2p$ elements for any prime $p>5$ \cite{HSV,McC}.

Our second observation about Table \ref{t:ldq} is that there are severe restrictions on the admissible orders of non-medial left distributive quasigroups. 
Many gaps are justified by the following theorem.

\begin{theorem}[\cite{ESG,Gra}]	
Every connected quandle with $p$ or $p^2$ elements, $p$ prime, is medial.
\end{theorem}

The prime case was proved by Galkin \cite{Gal-ldq} for quasigroups, and by Etingof, Soloviev and Guralnick \cite{ESG} for connected quandles. A conceptually simpler proof using canonical representation can be found in \cite[Section 8]{HSV}, here is an outline. First, use a group-theoretical result by Kazarin: in a finite group $G$, if $|a^G|$ is a prime power, then $\langle a^G\rangle$ is solvable; with little work, it follows that if $Q$ is a connected quandle of prime power size, then $\lmlt(Q)$ is solvable. Now recall that a transitive group (here: $\lmlt(Q)$) acting on a set of prime size (here: $Q$) is primitive, and apply a theorem of Galois stating that any finite solvable primitive group acts as a subgroup of the affine group over a finite field.

The prime square case for quasigroups is claimed by Galkin in \cite{Gal-survey} but never appeared in print; for connected quandles, it was solved by Gra\~na \cite{Gra}. For involutory left distributive quasigroups, the proof is substantially easier, see \cite{Nob1}. 
The prime cubed case is discussed in \cite{Bia}, but the classification is not easy to state. 

We can also observe that there are no non-medial left distributive quasigroups of order $2^k$ for $k=1,2,3,4,5$. However, this is not a general property: in fact, the first ever example of a left distributive quasigroup not isotopic to a Bol loop, constructed by Onoi \cite{Onoi-homog}, has $2^{16}$ elements. The smallest non-medial connected quandle with $2^k$ elements exists for $k=5$, but we do not know the smallest $k$ in the quasigroup case.

Our final observation is that there are precisely two non-medial left distributive quasigroups of order $3p$ for $p=5,7,11,13$. Two such examples were constructed for every prime $p\geq5$ by Galkin in \cite{Gal-small} (the construction was studied recently in a great detail in \cite{CEHSY,CH}, see also Example \ref{e:15}). It is an open problem whether there exist any other connected quandles with $3p$ elements.

\subsection{Structural properties}\label{ss:ldq-structure}

We will mention a few subalgebra and congruence properties here. A finite quasigroup of order $n$ has the \emph{Lagrange property}, if the order of every subquasigroup divides $n$. It has the \emph{Sylow property}, if, for every maximal prime power divisor $p^k$ of $n$, there is a subquasigroup of order $p^k$ (stronger versions of the Sylow property exist, and we refer to each particular paper for its own precise definition). Informally, a left distributive quasigroup is called \emph{solvable}, if it can be constructed by a chain of extensions by medial quasigroups; formal definitions differ \cite{Gal-sub,KNN,Nob6}, but they seem to share the following property: a left distributive quasigroup is solvable if and only if its left multiplication group is solvable.
(We note that it is not at all clear what is the ``correct" notion of solvability for quasigroups and loops, see \cite{SV} for a thorough discussion; the particular choice made by Glauberman, following Bruck, is only one of the reasonable options.)

Finite involutory left distributive quasigroups are solvable and have the Lagrange and Sylow properties. This has been proved independently several times, using each of the three methods we have discussed: through the conjugation representation in \cite{KNN}, through the isotopy to B-loops (combining Theorem \ref{thm:lsldq} and the results of Glauberman on B-loops \cite{Gla}), and through the homogeneous representation in \cite{Gal-solv}. In each case, the underlying group theoretical result is Glauberman's Z$^*$-theorem, which is used to show that the left multiplication group is solvable. An infinite counterexample to solvability is presented in \cite{Gal-solv}.  

Later, Galkin generalized the results into the non-involutory setting. 
In \cite{Gal-sub}, he proves that every finite solvable left distributive quasigroup has the Lagrange property, but not necessarily the Sylow property (a counterexample of order 15 exists). In \cite{Gal-sylow}, he proves the Sylow property under the additional assumption that the order of the quasigroup, and the order of its translations, are coprime (this is always true in the involutory case). 

Recall that all left distributive quasigroups isotopic to a group admit a homogeneous representation of the form $\Q(G,1,\psi)$, cf. Example \ref{e:b-module->q}(2).
They also satisfy the Lagrange and Sylow properties \cite[Theorem 5.3]{Gal-ldq}. This fact is used to show an important structural feature: a finite left distributive quasigroup with no non-trivial subquasigroups is medial \cite[Theorems 5.5 and 7.2]{Gal-ldq}.

More information about Galkin's results on left distributive quasigroups can be found in his survey paper \cite[Section 6]{Gal-survey}. A part of Galkin's theory was translated to English and clarified in~\cite{Vla}.

%In \cite{Nob8}, Nobusawa proves an analogy of Jordan-H\"older theorem for quandles.

%In \cite{Gal-sub}, Galkin investigated the correspondence between subquasigroups of $\Q(G,H,\psi)$ and subgroups of $G$. 

\section{Open problems}

Several interesting problems appeared to us while writing the paper. 

\subsection{Commutator theory over ``non-associative modules"}
Universal algebra develops a commutator theory based on the notion of \emph{abelianess}, related to affine representation over classical modules (see \cite{SV} for the commutator theory adapted to loops, and the references thereof). For instance, Theorem \ref{thm:medial} can be explained in this manner.
Is there a meaningful weakening of the principle of abelianess, related to affine representation over some sort of ``non-associative modules"? A one that would, for instance, explain Theorem \ref{thm:trimedial}? To what extent the module theoretic methods can be adapted to the non-associative setting?

\subsection{Non-idempotent generalization of left distributive quasigroups}

Find a ``non-idempotent generalization" of Theorem \ref{thm:ldq}: describe the class of quasigroups (whose idempotent members are precisely the left distributive quasigroups) that are right affine over Belousov-Onoi loops; perhaps, impose an additional condition on the representation in order to obtain an elegant description of the class. Theorem \ref{thm:trimedial} shall follow as an easy consequence of this generalization, just as it happens in the idempotent case (see Section \ref{ss:gbl}). We are not aware of any results even in the involutory case (generalizing Theorem \ref{thm:lsldq}).

\subsection{Enumeration}
The generic problem is, to extend all enumeration results presented in this paper. Perhaps the most interesting questions are:
\begin{enumerate}
	\item distributive and trimedial quasigroups of order $3^5$;
	\item commutative Moufang loops of order $3^6$ and the corresponding enumeration of distributive and trimedial quasigroups of order $3^6$;
	\item connected quandles and left distributive quasigroups of order $3p$, $p$ prime, or more generally, $pq$, $p,q$ primes;
	\item left distributive quasigroups of order $2^k$, $k>5$.
\end{enumerate}

\section*{Acknowledgement}

I am indebted to my former student Jan Vlach\'y for a thorough research on Galkin's papers and for explaining me their contents and significance. His remarkable student project \cite{Vla} on enumeration of small left distributive quasigroups convinced me that this is the right approach to left distributive quasigroups in particular, and connected quandles in general.

\newpage

\section*{Addendum: Finite left distributive quasigroup have solvable multiplication groups}

It is an outrageous ignorance that I missed the 2001 paper of Alexander Stein [S] in my survey (many thanks to Giuliano Bianco for pointing this out). Stein proves the following group-theoretical result, generalizing Glauberman's Z$^*$-theorem: Let $G$ be a finite group and $g\in G$ such that the conjugacy class $g^G$ is a transversal to some subgroup $H$ of~$G$. Then the subgroup $\langle g^G\rangle$ is solvable. The proof is complicated and uses the classification of finite simple groups. The following is an easy consequence.

\begin{theorem}[{[S, Theorem 1.4]}]
Let $Q$ be a finite left distributive quasigroup. Then $\lmlt(Q)$ is solvable.
\end{theorem}

\begin{proof}
Let $G=\lmlt(Q)$, $g=L_e$ for some $e\in Q$ and $H=\lmlt(Q)_e$, the stabilizer. Then $g^G=\{L_a:a\in Q\}$ is a transversal to $H$: indeed, $L_aH\cap g^G=\{L_a\}$, since $L_x\in L_aH$ iff $L_a^{-1}L_x\in H$ iff $ae=xe$ iff $a=x$ (here we need unique right division). Hence $G=\langle g^G\rangle$ is solvable.
\end{proof}

From the proof, we see that a quandle envelope $(G,\zeta)$ corresponds to a latin quandle if and only if $\zeta^G$ is a transversal to $G_e$. This seems to be an even more convenient characterization than the one presented on p. 22. 
A related argument also shows an interesting alternative to Proposition 4.2.

\begin{proposition}[{[S, Lemma 1.6]}]
Let $G$ be a finite group and $g\in G$ such that $g^G$ is a transversal to $C_G(g)$. Then the conjugation quandle over $g^G$ is latin.
\end{proposition}

Theorem 7.1 subsumes previous results in the involutory case [28,31,41] (based on the Z$^*$-theorem, see Section 6.3) and in the both-sided case [20] (Fischer's theorem, see Section 3.2).
As a corollary, using Galkin's results [25], we obtain that all finite left distributive quasigroups have the Lagrange property. 
%(It is however still unclear whether this is the ``correct" notion of solvability for quandles.)

Another short argument shows that all finite simple left distributive quasigroups are medial, hence affine over abelian groups.

\begin{corollary}
Finite simple left distributive quasigroups are medial.
\end{corollary}

\begin{proof}
An observation by Joyce [J, Proposition 3] says that if a quandle $Q$ is simple then $\lmlt(Q)'$ is the smallest normal subgroup of $\lmlt(Q)$. Since $\lmlt(Q)$ is solvable, we then must have $\lmlt(Q)'' = 1$, hence $\lmlt(Q)'$ is abelian, and so $Q$ is medial by [J, Remark on p. 308].
\end{proof}

The classification of finite simple medial quandles can be found in [J, Theorem 7], or [AG, Corollary 3.13].

%\begin{thebibliography}{99}
{\small

\bigskip
\noindent
[AG] N. Andruskiewitsch, M. Gra\~na, \emph{From racks to pointed Hopf algebras}, Adv. Math. 178/2 (2003), 177--243.

\smallskip
\noindent
[J] D. Joyce, \emph{Simple quandles}, Journal of Algebra 79 (1982), 307--318.

\smallskip
%\bibitem{Ste}
\noindent
[S] A. Stein, \emph{A conjugacy class as a transversal in a finite group}, Journal of Algebra 239 (2001), 365--390.
}
%\end{thebibliography}

\end{document}